\documentclass[11pt,a4paper,notitlepage,oneside]{amsart}
\usepackage{etoolbox}
\usepackage{microtype}
\usepackage{amsmath}
\usepackage{amsthm}
\usepackage{braket}
\usepackage{pdfpages}
\usepackage[psamsfonts]{amssymb}
\usepackage[T1]{fontenc}
\usepackage{lmodern}  
\usepackage[english]{babel}
\usepackage[utf8]{inputenc}
\usepackage{enumitem}
\usepackage{mathtools}
\usepackage{caption,subfig}
\usepackage{rotating}
\usepackage[makeroom]{cancel}
\usepackage[normalem]{ulem}
\usepackage{color,soul}
\usepackage{scalerel}
\usepackage{relsize}
\usepackage[normalem]{ulem}

\usepackage[autostyle,italian=guillemets]{csquotes}
\usepackage[style=alphabetic,firstinits=true,maxbibnames=99,backend=bibtex]{biblatex}
\allowdisplaybreaks 

\usepackage{pgf,tikz,pgfplots}

\usepackage{mathrsfs}
\usetikzlibrary{arrows}

\usepackage{fancyhdr}
\usepackage{tikz-cd}
\usepackage{amscd}
\usepackage[a4paper, bottom=3cm, left=3cm,right=3cm]{geometry}
\usepackage{hyperref}
\hypersetup{hidelinks}

\usepackage{subfiles}
\numberwithin{equation}{section} 
\theoremstyle{definition}
\newtheorem{defi}{Definition}[section]
\newtheorem{ex}[defi]{Example}

\theoremstyle{plain}
\newtheorem{thm}[defi]{Theorem}
\theoremstyle{plain}
\newtheorem*{thm*}{Theorem}
\theoremstyle{plain}

\theoremstyle{plain}

\theoremstyle{plain}
\newtheorem{prop}[defi]{Proposition}
\theoremstyle{plain}
\newtheorem{lemma}[defi]{Lemma}
\theoremstyle{plain}
\newtheorem{cor}[defi]{Corollary}
\theoremstyle{plain}

\theoremstyle{plain}
\newtheorem*{question*}{Question}

\theoremstyle{definition}
\newtheorem{rem}[defi]{Remark}
\theoremstyle{thm}

\definecolor{ao(english)}{rgb}{0.0, 0.5, 0.0}

\DeclareMathOperator{\Log}{\rm Log}

\DeclareMathOperator{\SL}{\rm SL}
\DeclareMathOperator{\SO}{\rm SO}
\DeclareMathOperator{\Sym}{\rm Sym}

\newcommand{\Mat}{\text{Mat}}
\newcommand{\argdot}{\,\cdot\,}
\newcommand{\di}{\text{d}}

\DeclareMathOperator{\rank}{\rm rk}

\newcommand{\ZZ}{\mathbb{Z}}
\newcommand{\RR}{\mathbb{R}}
\newcommand{\QQ}{\mathbb{Q}}
\newcommand{\CC}{\mathbb{C}}

\newcommand{\HH}{\mathbb{H}}

\newcommand{\weil}{\omega}

\DeclareMathOperator{\Gr}{\rm Gr}
\DeclareMathOperator{\Mp}{\rm Mp}

\newcommand{\lambdatwo}{\eta} 
\newcommand{\primenum}{\mathfrak{p}} 
\newcommand{\brK}{M} 
\newcommand{\prol}{\pi} 

\newcommand{\intfunct}{{\mathcal{I}_{\aalpha}}} 

\DeclareMathOperator{\KMliftbase}{\Lambda_{\rm KM}}
\DeclareMathOperator{\KMliftbasel}{\Lambda_{\rm KMF,\ell}}

\newcommand{\hermdom}{\mathcal{D}}
\newcommand{\KMoper}{\mathcal{D}^{p,q}}

\newcommand{\FFa}{F_{\aalpha}}

\newcommand{\hha}{h_{\aalpha}}

\newcommand{\sqH}[1]{H^{#1}_{sq}} 
\newcommand{\locsys}[2]{\widetilde{\Sym}{}^{#1}(#2)} 
\newcommand{\wrep}{\weil} 
\newcommand{\primen}{\primenum}
\newcommand{\bbeta}{\underline{\indextwo}} 
\newcommand{\bbetau}{\overline{\beta}} 
\newcommand{\ggammau}{\overline{\gamma}} 
\newcommand{\FFg}{F^{\ggammau}}
\newcommand{\intfunctg}{{\mathcal{I}^{\ggammau}}} 
\newcommand{\indextwo}{\beta} 
\newcommand{\basee}{\mathfrak{e}} 
\newcommand{\dwrep}[1]{\rho_{#1}} 
\newcommand{\aalpha}{\underline{\alpha}} 
\newcommand{\aalphau}{\overline{\alpha}} 
\newcommand{\genU}{u} 
\newcommand{\genUU}{u'} 
\newcommand{\genvec}{v} 
\newcommand{\basevec}{e} 

\newcommand{\Gpol}{\mathcal{Q}} 
\newcommand{\pol}{\mathcal{P}} 
\newcommand{\polw}[3]{\mathcal{P}_{#1,#2,#3}} 

\newcommand{\borw}{g^\#}
\newcommand{\boralpha}{\delta}
\newcommand{\borbeta}{\nu}
\newcommand{\borgamma}{\epsilon} 
\newcommand{\borK}{M} 

\DeclareMathOperator{\bigO}{\rm O}

\def\be{\begin{equation}}
\def\ee{\end{equation}}
\def\bes{\begin{equation*}}
\def\ees{\end{equation*}}
\def\ba{\be\begin{aligned}}
\def\ea{\end{aligned}\ee}
\def\bas{\bes\begin{aligned}}
\def\eas{\end{aligned}\ees}

\author{Ingmar Metzler}
\address{Department of Mathematics, ETH Zurich, 8092 Zurich, Switzerland}
\email{ingmar.metzler@math.ethz.ch}

\author{Riccardo Zuffetti}
\address{
Fachbereich Mathematik, Technische Universität Darmstadt, Schlossgartenstraße 7, D–
64289 Darmstadt, Germany.
}
\email{zuffetti@mathematik.tu-darmstadt.de}

\title{Injectivity of the genus 1 Kudla--Millson lift on locally symmetric spaces}

\newcommand{\myblue}{black}

\begin{document}
	\maketitle
	\begin{abstract}
	Let~$L$ be an even indefinite lattice.
	We show that if~$L$ splits off a hyperbolic plane and a scaled hyperbolic plane, then the Kudla--Millson lift of genus~$1$ associated to~$L$ is injective.
	Our result includes as special cases all previously known injectivity results on the whole space of elliptic cusp forms available in the literature.
	In particular, we also consider the Funke--Millson twist of the lift.
	Further, we provide geometric applications on locally symmetric spaces of orthogonal type.
	\end{abstract}
	\tableofcontents

	\section{Introduction}
	
	In their celebrated article~\cite{hirzebruchzagier}, Hirzebruch and Zagier proved that the generating series of intersection numbers of special divisors on Hilbert modular surfaces is a modular form.
	This result has been generalized by Kudla and Millson \cite{kumi;harmI}, \cite{kumi;harmII}, \cite{kumi;intnum} to the locally symmetric spaces arising from~$\bigO(p,q)$.
	They constructed a theta function that behaves as a modular form in one variable and is a differential form on the symmetric space in the second variable.
	Further, they proved that the \emph{cohomology class} of such a differential form behaves as a \emph{holomorphic modular form}.
	In fact, this class equals the generating series of cohomology classes of the so-called \emph{special cycles} on the locally symmetric space.
	
	Information on the behavior of the Kudla--Millson theta function provides geometric data of the associated locally symmetric space.
	For instance, Funke and Millson~\cite{fm;boundarybeha} proved the non-vanishing of the (cohomology class of the) Kudla--Millson theta function, deducing that infinitely many special cycles have non-vanishing cohomology classes.
	Funke and Kudla~\cite{funkekudla} obtained a ``mock theta'' decomposition of the theta integrals obtained by integrating compactly supported forms against the Kudla--Millson theta function.
	Restricting to the hyperbolic case, i.e.~assuming~$q=1$, they recovered Zwegers' mock theta function~\cite{zwegers}.
	Under the assumption that the locally symmetric space admits a complex structure, i.e.~either~$p=2$ or~$q=2$, several properties on cones of special cycles have been illustrated~\cite{brmo}, \cite{zufcones}, \cite{zufequidistr}.
	
	Instead of integrating the Kudla--Millson theta function against a compactly supported differential form on the orthogonal side, it is possible to pair it against a cusp form on the symplectic side by means of the Petersson inner product.
	We refer to this linear map from cusp forms to closed differential forms as the \emph{Kudla--Millson lift}.
	
	The \emph{injectivity} of the Kudla--Millson lift is of interest for several reasons.
	For instance, it implies a converse theorem for Borcherds' products~\cite{brfutwo}, \cite{brfu}, \cite{bruinier;converse}, and provides a way of computing the cohomology groups of orthogonal Shimura varieties~\cite{blmm;conj}, \cite{kieferzuffetti}.
	
	The known injectivity results in the Hermitian case are quite general and cover many cases of interest.
	For general signature~$(p,q)$, however, the overall behavior of the lift is far from being well-understood.
	The first results in this direction were given by Bruinier and Funke in~\cite{brfu}, where the	injectivity has been proven under the assumption that the lattice giving rise to the locally symmetric space is \emph{unimodular}.
	Recently, Stein~\cite{stein} proved that such a lift is injective also if the lattice is \emph{maximal} and of sufficiently large rank. 
	
	The goal of this paper is to extend the injectivity results in general signature to the same level of generality available for the Hermitian case.
	We do so by implementing the strategy given by the second author in~\cite{zufunfolding} to the general signature case.
	Our result, which we make explicit below, improves all previously known injectivity results on the space of elliptic cusp forms.
	
	We denote the hyperbolic plane by~$U$, and by~$U(N)$ the lattice~$U$ whose quadratic form is rescaled by some positive integer~$N$.
	\begin{thm}\label{thm;mainres}
		Let $L$ be an even indefinite lattice of signature~$(p,q)$.
		\begin{enumerate}[label=(\roman*)]
			\item If~$L \cong M \oplus U(N) \oplus U$ for some even lattice~$M$ and some positive integer~$N$, then the Kudla--Millson lift associated to~$L$ is injective.
			\item Let~$q=1$.
			If~$L \cong M \oplus U$ for some positive definite lattice~$M$ and~$M\otimes\ZZ_\primen$ splits off a hyperbolic plane over~$\ZZ_\primen$ for every prime~$\primen \in \ZZ$, then the lift is injective.
			\item Let~$p=1$.
			Then the Kudla--Millson lift associated to~$L$ is identically zero.
		\end{enumerate}
	\end{thm}
	
	We provide below an overview of the main ideas of the proof of Theorem~\ref{thm;mainres}, as well as some geometric applications.
	Our result covers also the case of lifts with local coefficients, given by some symmetric power~$\Sym^\ell(L\otimes\RR)$ as introduced by Funke and Millson~\cite{fm;cycleswith}.
	This is of interest also in~\cite{brfu} and~\cite{stein}.
	For the sake of simplicity, in this introduction we do not treat such a generalization, i.e.\ we restrict to the case~$\ell=0$, and instead refer to Section~\ref{sec;caseofsymmpowers} for details. 
	
	\subsection{Overview of the main result}
	Let~$L$ be an indefinite even lattice of signature~$(p,q)$.
	We denote by~$(\argdot{,}\argdot)$ the bilinear form on~$L$, and by~$q(\argdot)=(\argdot{,}\argdot)/2$ the associated quadratic form.
	The discriminant group of~$L$ is the quotient~$L'/L$, where~$L'$ is the dual lattice of~$L$.
	The quadratic form~$q$ induces a~$\QQ/\ZZ$-valued quadratic form on~$L'/L$.
	
	The metaplectic double cover~$\Mp_2(\ZZ)$ of~$\SL_2(\ZZ)$ acts on the group algebra~$\CC[L'/L]$ under the Weil representation~$\dwrep{L}$ associated to~$L$.
	This representation is unitary with respect to the standard scalar product~$\langle\argdot{,}\argdot\rangle$ of~$\CC[L'/L]$.
	We denote by~$M^k_L$ and~$S^k_L$ the space of modular forms, respectively cusp forms, of weight~$k$ with respect to~$\dwrep{L}$; see Section~\ref{sec;backgroundvvmodforms} for details.
	Every~$f\in M^k_L$ has a Fourier expansion of the form
	\[
	f(\tau)
	=
	\sum_{\borgamma\in L'/L}
	\sum_{\substack{n\in\ZZ+q(\borgamma)\\ n\ge0}}
	c_n(f_\borgamma)q^n\mathfrak{e}_\borgamma,\qquad\tau\in\HH,
	\]
	where~$q=e^{2\pi i \tau}$ and~$\mathfrak{e}_\borgamma$ is the standard basis vector of~$\CC[L'/L]$ associated to~$\borgamma\in L'/L$.
	
	Let~$V=L\otimes\RR$ and~$G=\SO(V)\cong\SO(p,q)$.
	The symmetric domain~$\hermdom$ associated to~$G$ is the Grassmannian of negative-definite subspaces of dimension~$q$ in~$V$.
	A \emph{locally symmetric space} arising from~$L$ is a quotient space~$X=\Gamma\backslash\hermdom$, for some finite index subgroup~$\Gamma$ of~$\SO(L)$.
	We assume that~$\Gamma$ is torsion-free, so that~$X$ inherits a structure of real manifold of dimension~$pq$.
	
	Kudla and Millson~\cite{kumi;harmI}, \cite{kumi;harmII} constructed a~$G$-invariant Schwartz form~$\varphi_{\text{\rm KM}}$ on~$V$ with values in the space~$\mathcal{Z}^q(\hermdom)$ of smooth closed~$q$-forms on~$\hermdom$.
	Such a Schwartz function may be rewritten in terms of the standard Gaussian of~$V$ and some products of Hermite polynomials of total degree~$q$.
	We provide details on~$\varphi_{\text{\rm KM}}$ in Section~\ref{sec;compDp2}.
	
	Let~$k=(p+q)/2$.
	The \emph{Kudla--Millson theta function} associated to~$L$ is defined as
	\[
	\Theta(\tau,z,\varphi_{\text{\rm KM}})
	= y^{-k/2} \sum_{\borgamma \in L'/L} \sum_{\lambda\in L + \borgamma}\big(\weil(g_\tau)\varphi_{\text{\rm KM}}\big)(\lambda,z)
	\mathfrak{e}_\borgamma
	\]
	for every~$\tau=x+iy\in\HH$ and~$z\in\hermdom$, where~$\weil$ is (the Schrödinger model of) the Weil representation. 
	Such a theta function behaves as a (non-holomorphic) modular form of weight~$k$ with respect to the variable~$\tau$, and is a closed~$q$-form on~$\hermdom$ with respect to the variable~$z$.
	
	The {Kudla--Millson lift} associated to~$L$ is the linear function mapping cusp forms in~$S^k_L$ to their Petersson inner product with~$\Theta(\tau,z,\varphi_{\text{\rm KM}})$, namely
	\[
	\KMliftbase\colon S^k_L\longrightarrow\mathcal{Z}^g(\hermdom),
	\qquad
	f\longmapsto\int_{\SL_2(\ZZ)\backslash\HH}y^k\langle f(\tau),\Theta(\tau,z,\varphi_{\text{\rm KM}})\rangle \frac{dx\,dy}{y^2}.
	\]
	
	It is shown in~\cite{zufunfolding} that if~$p=2$, then~$\Theta(\tau,z,\varphi_{\text{\rm KM}})$ may be rewritten in terms of the Siegel theta functions constructed by Borcherds in~\cite{bo;grass}.
	This served as the starting point of the new proof of the injectivity of~$\KMliftbase$ in the case of \emph{orthogonal Shimura varieties}; see~\cite[Theorem~$7.1$]{zufunfolding}.
	
	In this article we show that a rewriting of~$\Theta(\tau,z,\varphi_{\text{\rm KM}})$ in terms of Siegel theta functions is possible also in general signature, as illustrated below.
	If~$\pol$ is a homogeneous polynomial on~$\RR^{p+q}$, then we denote by~$\Theta_L(\tau,g,\pol)$ the Siegel theta function associated to~$L$ and~$\pol$ constructed by Borcherds~\cite{bo;grass}.
	This is a function of~$\tau\in\HH$ and~$g\in G$; see Section~\ref{sec;Siegthetafu} for a brief introduction.
	In this paper we construct homogeneous polynomials~$\pol_{\aalpha}$ of degree~$q$ such that
	\[
	\Theta(\tau,z,\varphi_{\text{\rm KM}})
	=
	\sum_{\aalpha}\Theta_L(\tau,g,\pol_{\aalpha})
	\otimes
	g^*\big(\omega_{\aalpha}\big).
	\]
	Here~$\aalpha=(\alpha_1,\dots,\alpha_q)$ is a tuple of indices with entries~$\alpha_j\in\{1,\dots,p\}$, the isometry~$g\in G$ is chosen to map~$z$ to a fixed base point~$z_0$ of~$\hermdom$, and~$\omega_{\aalpha}\in\bigwedge^q T^*_{z_0}\hermdom$ are \emph{linearly independent} vectors arising from the very definition of the Kudla--Millson Schwartz function.
	We then deduce that
	\be\label{eq;introliftdecSiegel}
	\KMliftbase(f)
	=
	\sum_{\aalpha}
	\bigg(
	\underbrace{\int_{\SL_2(\ZZ)\backslash\HH}y^{k}
	\big\langle f(\tau),\Theta(\tau,g,\pol_{\aalpha})\big\rangle \frac{dx\,dy}{y^2}}_{\intfunct(g)}
	\bigg)
	\cdot g^*(\omega_{\aalpha}).
	\ee
	We refer to the functions~$\intfunct\colon G\to\CC$ appearing in~\eqref{eq;introliftdecSiegel} as the \emph{defining integrals of the Kudla--Millson lift of~$f$}.
	
	Since the differential form appearing on the right-hand side of~\eqref{eq;introliftdecSiegel} is pointwise written as a combination of linearly independent vectors in~${\bigwedge}^qT^*_z\hermdom$, we deduce that the lift~$\KMliftbase(f)$ vanishes if and only if its defining integrals~$\intfunct$ vanish.

	In this paper we apply the unfolding method of Borcherds~\cite[Section~$5$]{bo;grass} to compute the Fourier expansion of~$\intfunct$. 
	For the sake of simplicity, we assume in this introduction that~$L = M \oplus U$ splits a hyperbolic plane $U$  
	with standard generators $\genU$ and $\genUU$. 
	If~$\genvec\in V$ and~$z$ is a subspace of~$V$, then we denote by~$\genvec_z$ the orthogonal projection of~$\genvec$ to~$z$.
	For every~$z\in\hermdom$, we denote by~$w^\perp$ the orthogonal complement of~$u_{z^\perp}$ in~$z^\perp$.
	We refer to Section~\ref{sec;redsmallerlat} for details on the twist~$\borw$ of~$g\in G$ and the decomposition of~$\pol_{\aalpha}$ in terms of polynomials~$\pol_{\aalpha,\borw,h^+,h^-}$ defined on subspaces of~$V$.
	
	\begin{thm}\label{thm;introFourExp}
	Let~$f\in S^k_L$.
	The defining integrals~$\intfunct$ of the Kudla--Millson lift of~$f$ admit a Fourier expansion of the form
	\[
	\intfunct(g)
	=
	\sum_{\substack{\lambda\in\brK' \\ q(\lambda)\ge 0}} c_\lambda(g) e^{2\pi i (\lambda,\mu)},
	\]
	where~$\mu=-\genUU+\genU_{z^\perp}/2\genU_{z^\perp}^2+\genU_z/2\genU_z^2$.
	If~$\lambda\in\brK'$ is of positive norm, then the Fourier coefficient of index~$\lambda$ of~$\intfunct$ is given by
	\bas
	c_\lambda(g)
	&=
	\frac{\sqrt{2}}{|\genU_{z^\perp}|}
	\sum_{\substack{t \geq 1 \\ t \mid \lambda}} 
		\sum_{h^+=0}^q
		\bigg(\frac{t}{2i}\bigg)^{h^+}
		c\big(f_{\lambda/t},q(\lambda)/t^2\big)
		\int_{0}^{\infty}
		y^{q + (p-5)/2 - h^+}
		\\
		& \quad\times
		\exp\bigg(-\frac{2\pi y \lambda_{w^\perp}^2}{t^2} - \frac{\pi t^2}{2 y\genU_{z^\perp}^2 } \bigg)
		\exp(-\Delta /8\pi y)(\polw{\aalpha,\borw}{h^+}{0})\big(\borw(\lambda/t)\big) dy.
	\eas
	\end{thm}
	The main result of the present paper, namely Theorem~\ref{thm;mainres}, follows from Theorem~\ref{thm;introFourExp}.
	The idea, for the sake of brevity here illustrated under the assumption that~$p>1$, is as follows.
	Suppose that~$f\in S^k_L$ is such that~$\KMliftbase(f)=0$.
	As we remarked above, this implies that the defining integrals~$\intfunct$ of~$\KMliftbase(f)$, and hence their Fourier coefficients, vanish.
	We show that
	\be\label{eq;vanishlatz}
	c(f_{\borgamma},q(\lambda))=0\qquad\text{for every~$\borgamma\in M'/M$ and~$\lambda\in M+\borgamma$ such that~$q(\lambda)>0$}.
	\ee
	This is deduced from Theorem~\ref{thm;introFourExp} by induction on the divisibility of~$\lambda$.
	By applying the newform theory developed by Bruinier, we deduce from~\eqref{eq;vanishlatz} that all Fourier coefficients of~$f$ vanish under the assumptions of Theorem~\ref{thm;mainres}.
	
	\subsection{Some geometric applications}
	The locally symmetric space~$X=\Gamma\backslash\hermdom$ is a complete Riemannian manifold of finite volume.
	We denote by~$H^*(X,\CC)$ and~$\sqH{*}(X,\CC)$ respectively the de Rham and the~$L^2$-cohomology of complex-valued differential forms on~$X$.
	The inclusion of the space of square-integrable forms in the space of smooth forms induces a homomorphism
	\be\label{eq;homocoho}
	\sqH{r}(X,\CC) \longrightarrow H^r(X,\CC)
	\ee
	for every~$r\ge0$.
	In general, the map~\eqref{eq;homocoho} is neither injective nor surjective.
	However, Zucker~\cite{zucker} proved that there exists a constant~$c_G$ such that~\eqref{eq;homocoho} is an isomorphism for all~$r\le c_G$.
	Under this assumption, by Hodge theory, the group~$\sqH{r}(X,\CC)$ is isomorphic to the space of~$L^2$ harmonic $r$-forms; cf.~\cite[Section~$2.4$]{blmm;conj}.
	The lift of a cusp form under~$\KMliftbase$ is a \emph{harmonic}~$q$-form by~\cite[Theorem~$4.1$]{kumi;tubes}, as well as a square-integrable form by~\cite[Proposition~$4.1$]{brfu}.
	These, together with our injectivity result in Theorem~\ref{thm;mainres}, imply the following.
	\begin{cor}\label{cor:boundcoho}
	Let~$X$ be a locally symmetric space arising from an even indefinite lattice~$L$ of signature~$(p,q)$, with~$p>1$, satisfying the hypothesis of Theorem~\ref{thm;mainres}.
	If~$q\leq c_G$, where~$c_G$ is Zucker's constant, then~$\dim H^q(X,\CC)\ge \dim S^k_L$.
	\end{cor}
	The case of~$q=1$ is of special interest.
	In this case~$X$ is said to be a \emph{hyperbolic manifold}.
	Since~$c_G=\lfloor\frac{p}{2}\rfloor-1$ by~\cite[Theorem~$6.2$]{zucker}, we may refine Corollary~\ref{cor:boundcoho} as follows.
	\begin{cor}
	Let~$X$ be a hyperbolic manifold arising from an even lattice~$L$ of signature~$(p,1)$.
	Suppose that~$L$ splits off a hyperbolic plane~$U$, and the positive definite sublattice~$\brK=U^\perp\subset L$ is such that~$\brK\otimes\ZZ_\primen$ splits off a hyperbolic plane for every prime~$\primen$.
	If~$p\ge4$, then~$\dim H^1(X,\CC)\ge \dim S^k_L$.
	\end{cor}
	
	\subsection*{Acknowledgments}
	We are grateful to Jan Bruinier, Jens Funke, Yingkun Li, Paul Kiefer and Oliver Stein for several useful conversations.
	The authors are partially supported by the Collaborative Research Centre TRR~$326$ ``Geometry and Arithmetic of Uniformized Structures'', project number~$444845124$.
	
	\section{The Kudla--Millson Schwartz function}\label{sec;compDp2}
	
	Let~$V$ be a real vector space endowed with a symmetric bilinear form~$(\argdot{,}\argdot)$ of signature~$(p,q)$.
	Its associated quadratic form is defined as $q(\argdot)=(\argdot{,}\argdot)/2$.
	In this section, we provide an explicit formula of the Kudla--Millson Schwartz function~$\varphi_{\text{\rm KM}}$ attached to~$V$, following the exposition of~\cite[Section~$2$ and Section~$4$]{brfutwo} and~\cite[Section~$7$]{ku;algcycle}.\\
	
	We fix once and for all an orthogonal basis~$(\basevec_j)_j$ of~$V$ such that $(\basevec_\alpha,\basevec_\alpha)=1$ for every $\alpha=1,\dots,p$, and~${(\basevec_\mu,\basevec_\mu)=-1}$ for $\mu=p+1, \hdots, p+q$.
	We denote the corresponding coordinate functions by~$x_\alpha$ and~$x_\mu$.
	The choice of the basis~$(\basevec_j)_j$ is equivalent to the choice of an isometry~${g_0\colon V\to\RR^{p,q}}$, where~$\RR^{p,q}$ is the real space~$\RR^{p+q}$ endowed with the standard quadratic form of signature~$(p,q)$.
	
	To $V$ we associate the \emph{Grassmannian} 
	\bes
	\Gr(V)=\{z\subset V : \text{$\dim z=q$ and $(\argdot{,}\argdot)|_z<0$}\} 
	\ees
	with \emph{base point}~$z_0$ spanned by~$(\basevec_{j})_{j=p+1}^{p+q}$.
	The symmetric space arising as~$\hermdom=G/K$, where~$G=\SO(V)$ and~$K$ is the maximal compact subgroup of~$G$ stabilizing $z_0$, may be identified with~$\Gr(V)$; see~\cite[Part~$2$, Section~$2.4$]{1-2-3}. 
	It carries a natural complex structure only if $q=2$ or $p=2$. 
	From now on, we write~$\hermdom$ and~$\Gr(V)$ interchangeably.
	
	For every $z\in \hermdom$, we define the \emph{standard majorant} $(\argdot{,}\argdot)_z$ as
	\be\label{eq;standardmajorant}
	(\genvec,\genvec)_z=(\genvec_{z^\bot},\genvec_{z^\bot})-(\genvec_z,\genvec_z),
	\ee
	where $\genvec=\genvec_z+\genvec_{z^\bot}\in V$ is rewritten with respect to the decomposition~$V=z\oplus z^\bot$.
	
	Let~$\mathfrak{g}$ and~$\mathfrak{k}$ be the Lie algebras of~$G$ and~$K$ respectively, and let $\mathfrak{g}=\mathfrak{p}+\mathfrak{k}$ be the associated Cartan decomposition.
	It is well-known that~$\mathfrak{p}\cong\mathfrak{g}/\mathfrak{k}$ is isomorphic to the tangent space of~$\hermdom$ at the base point~$z_0$.
	With respect to the basis of $V$ chosen above, we have
	\be\label{eq;mathfracpiso}
	\mathfrak{p}\cong\left\{\left(
	\begin{smallmatrix}
	0 & X\\
	X^t & 0
	\end{smallmatrix}
	\right) : X\in\Mat_{p,q}(\RR)\right\}\cong \Mat_{p,q}(\RR).
	\ee
	
	To simplify the notation, we set~$e(t)=\exp(2\pi i t)$ for every~$t\in\CC$, and denote by~${\sqrt{t}=t^{1/2}}$ the principal branch of the square root, so that~$\arg(\sqrt{t})\in(-\pi/2,\pi/2]$.
	If~$s\in\CC$, define~${t^s=e^{s\Log(t)}}$, where~$\Log(t)$ is the principal branch of the logarithm.
	
	We realize the elements of the metaplectic double cover~$\Mp_2(\RR)$ of~$\SL_2(\RR)$ by couples of the form~$(\gamma,\phi(\tau))$, where~$\gamma=\big(\begin{smallmatrix}
			a & b\\
			c & d
		\end{smallmatrix}\big)\in\SL_2(\RR)$ and~$\phi\colon\HH\to\CC$ is a holomorphic function such that~$\phi(\tau)^2=c\tau+d$.
		We consider the elements of~$\SL_2(\RR)$ as elements of~$\Mp_2(\RR)$ by choosing~$\phi(\tau)=\sqrt{c\tau+d}$.

	The usual action~$\tau\mapsto \gamma\cdot\tau = (a\tau + b)/(c\tau + d)$ of~$\SL_2(\RR)$ on~$\HH$ under the Möbius transformation induces an action of~$\Mp_2(\RR)$ on~$\HH$.
	The group law of~$\Mp_2(\RR)$ is given by
	\[
	\big(\gamma_1,\phi_1(\tau)\big)\cdot \big(\gamma_2,\phi_2(\tau)\big)
	=
	\Big(\gamma_1\gamma_2,\phi_1(\gamma_2\cdot\tau)\phi_2(\tau)\Big),
	\qquad\text{for~$(\gamma_j,\phi_j)\in\Mp_2(\RR)$.}
	\]
	\begin{defi}\label{defi;schrmod}	
	We denote by $\wrep$ the Schrödinger model of (the restriction of) the Weil representation 
	of~$\Mp_2(\RR)\times \bigO(V)$ acting on the space~$\mathcal{S}(V)$ of Schwartz functions on~$V$.
	The action of~$\bigO(V)$ is defined as
	\bes
	\wrep(g)\varphi(\genvec)=\varphi\big(g^{-1}(\genvec)\big)
	\ees
	for every~$\varphi\in\mathcal{S}(V)$ and~$g\in \bigO(V)$.
	The action of~$\Mp_2(\RR)$ is given by
	\ba\label{eq;actMp2Schrod}
	\wrep\left(\left(\begin{smallmatrix}
	1 & x\\ 0 & 1
	\end{smallmatrix}\right),1\right)\varphi(\genvec)&=e(x q(\genvec))\varphi(\genvec)\quad\text{for every $x\in\RR$},\\
	\wrep\left(\left(\begin{smallmatrix}
	a & 0\\ 0 & a^{-1}
	\end{smallmatrix}\right),a^{-1/2}\right)\varphi(\genvec)&=a^{(p+q)/2}\varphi(a\genvec),\quad\text{for every $a>0$}\\
	\wrep(S)\varphi(\genvec)&=\sqrt{i}^{p-q}\widehat{\varphi}(-\genvec),
	\ea
	where 
	$S=\big(\big(
		\begin{smallmatrix}
		0 & -1\\ 
		1 & 0
		\end{smallmatrix}
	\big),\tau^{1/2}\big)$, 
	and $\widehat{\varphi}(\xi)=\int_V\varphi(\genvec)e^{2\pi i (\genvec,\xi)}\di v$ is the Fourier transform of $\varphi$.
	\end{defi}
	
	The \emph{standard Gaussian of}~$\RR^{p,q}$ is defined as
	\bes
	\varphi_0(x_1,\dots,x_{p+q})=\exp\Big(\!-\pi \, {\textstyle\sum}_{j=1}^{p+q} \, x_j^2\Big),\qquad\text{for every $(x_1,\dots,x_{p+q})^t\in\RR^{p+q}$}.
	\ees
	The \emph{standard Gaussian of} $V$ is the Schwartz function~$\varphi_0$ as above evaluated on the coordinate functions~$x_1,\dots,x_{p+q}$ of $V$ with respect to the chosen basis $(\basevec_j)_j$.
	Note that we may rewrite it by means of the standard majorant of the base point~$z_0\in\Gr(V)$ as~$\exp\big(-\pi (\argdot {,}\argdot)_{z_0}\big)$.
	It is~$K$-invariant with respect to the action of $\wrep$. 
	
	We denote by $\mathcal{S}(V)^K$ the space of $K$-invariant Schwartz functions on~$V$ and note that
	\be\label{eq;isoschwKinschwKDinf}
	\mathcal{S}(V)^K\cong\big[\mathcal{S}(V)\otimes C^\infty (\hermdom)\big]^G,
	\ee
	where the isomorphism is given by evaluating at the base point $z_0\in \hermdom=\Gr(V)$.

	\begin{ex}
	The isomorphism~\eqref{eq;isoschwKinschwKDinf} maps the standard Gaussian~$\varphi_0\in\mathcal{S}(V)^K$ to $\varphi_0(\genvec,z)=e^{-\pi(\genvec,\genvec)_z}$, where $(\argdot{,}\argdot)_z$ is the standard majorant defined in~\eqref{eq;standardmajorant}.
	\end{ex}

	Next, we define the Kudla--Millson Schwartz function $\varphi_{\text{\rm KM}}\in\big[\mathcal{S}(V)\otimes \mathcal{A}^q (\hermdom)\big]^G$, where we denote by~$\mathcal{A}^q(\hermdom)$ the space of~$q$-forms on~$\hermdom$.
	We remark that
	\be\label{eq;isoschwartzbigwedge}
	\big[\mathcal{S}(V)\otimes \mathcal{A}^q(\hermdom)\big]^G\cong\big[\mathcal{S}(V)\otimes{\bigwedge}^q(\mathfrak{p}^*)\big]^K.
	\ee
	Therefore, a $G$-invariant element $\varphi\in \mathcal{S}(V)\otimes \mathcal{Z}^q(\hermdom)$ may be defined as an element of $\big[\mathcal{S}(V)\otimes{\bigwedge}^q(\mathfrak{p}^*)\big]^K$ and extended to all of~$\hermdom$ via the action of~$G$.
	We follow this idea to define~$\varphi_{\text{\rm KM}}$.
	
	We denote by $X_{\alpha,\mu}$, with $1\le\alpha\le p$ and~$1\le\mu\le q$, the basis elements of~$\Mat_{p,q}(\RR)$ given by matrices with~$1$ at the~$(\alpha,\mu)$-th entry and zero otherwise.
	These elements provide a basis of~$\mathfrak{p}$ under the isomorphism~\eqref{eq;mathfracpiso}.
	Let~$\omega_{\alpha,\mu}$ be the element of the dual basis which extracts the $(\alpha,\mu)$-th coordinate of elements in~$\Mat_{p,q}(\RR)$, and let~$A_{\alpha,\mu}$ denote the left multiplication by~$\omega_{\alpha,\mu}$.
	\begin{defi}\label{defi;KMfunct}
	The function~$\varphi_{\text{\rm KM}}$ is defined by applying the operator
	\be\label{eq;KMOperator}
	\KMoper 
	\coloneqq \frac{1}{2^{q/2}}\prod_{\mu=1}^{q}\bigg[\sum_{\alpha=1}^p\bigg(x_\alpha-\frac{1}{2\pi}\frac{\partial}{\partial x_\alpha}\bigg)\otimes A_{\alpha,\mu}\bigg]
	\ee
	to the standard Gaussian $\varphi_0\otimes 1\in \big[\mathcal{S}(V)\otimes{\bigwedge}^0(\mathfrak{p}^*)\big]^K$.
	\end{defi}

	\subsection{On Hermite polynomials}
	The Kudla--Millson Schwartz function may be rewritten in terms of products of Hermite polynomials.
	For later purpose, we recall here some formulas of such polynomials.
	
	\begin{defi}\label{de;HermitePolynomials}
		Let~$q\in \ZZ_{\geq 0}$.
		The~$q$-th Hermite polynomial~$H_q(x)$ is defined as
		\bes
		H_q(t)=(-1)^q e^{x^2} \frac{d^q}{dx^q}e^{-x^2}.
		\ees
	\end{defi}
	It is well-known that such polynomials satisfy Rodrigues' formula
	\be\label{eq;rodriguesformula}
	H_q(x)=\Big(2x-\frac{d}{dx}\Big)^q\cdot 1.
	\ee
	By induction we may infer  
	\ba\label{eq;Hermitepolynomialexplicit}
	H_q(x) = 
	q! \sum_{l=0}^{\lfloor q/2\rfloor} \frac{(-1)^{l}}{(l)! (q-2l)!} \cdot (2x)^{q-2l}.
	\ea
	As we will see below, these polynomials form building blocks for the Kudla--Millson Schwartz function, due to the following property; compare with \cite[Section V.5]{br;Szego}.

	\begin{lemma}\label{lemma;KMPolynomialonevariable}
		Let~$q \in \ZZ_{\geq 0}$.
		We have 
		\ba
		\Big(x-\frac{1}{2\pi}\frac{d}{dx}\Big)^q e^{- \pi x^2} = 
		\frac{e^{- \pi x^2}}{(2\pi)^{q/2}} H_q(\sqrt{2\pi}x).
		\ea
	\end{lemma}

	\subsection{A formula for the Schwartz function}
	
	We now compute~$\varphi_{\text{\rm KM}}$ explicitly.
	We may rewrite
	\begin{align*}\label{eq;firstcompDp2}
	\varphi_{\text{\rm KM}}
	&=
	\KMoper(\big(\varphi_0\otimes 1\big)
	=
	\frac{1}{2^{q/2}}\Big[\sum_{\aalpha} \prod_{\mu = 1}^{q} \Big(\underbrace{\Big(x_{\alpha_\mu}-\frac{1}{2\pi}\frac{\partial}{\partial x_{\alpha_\mu}}\Big)\otimes A_{{\alpha_\mu}, \mu}}_{(\dagger_{\aalpha,\mu})}\Big)\Big]\big(\varphi_0\otimes 1\big).
	\end{align*}
	Here~$\aalpha$ denotes a collection $(\alpha_1,\hdots,\alpha_q)$ with $\alpha_i \in \{1,\hdots,p\}$. 
	Since the product above
	of the operators~$(\mathsmaller{\dagger_{\aalpha,\mu}})$ is applied componentwise, we deduce that
	\be\label{eq;KMschwrew}
		\varphi_{\text{\rm KM}} 
		=
		\sum_{\aalpha}\underbrace{\frac{1}{2^{q/2}} \prod_{\mu= 1}^{q} \Big(x_{\alpha_\mu} - \frac{1}{2\pi}\frac{\partial}{\partial x_{\alpha_\mu}}\Big) \varphi_0}_{(\star)}\otimes\, \bigwedge_{\mu= 1}^{q} \omega_{\alpha_\mu, \mu}.
	\ee

		\begin{rem}\label{rem;KMsFunctionpart}
	The operator $(\star)$ does not primarily depend on the tuple $\aalpha = (\alpha_1, \hdots, \alpha_q)$ with values in $\{1, \hdots, p\}$, but on the size of its fibers. 
			In other words, if we associate to~$\aalpha$ a tuple $\overline{\alpha} = (q_1, \hdots, q_p)$ where $q_j \coloneqq |\aalpha^{-1}(j)|$ is the multiplicity of the value~$j$ attained by~$\aalpha$, then by Lemma~\ref{lemma;KMPolynomialonevariable} we obtain 
			\be\label{eq;melpanh}
			\frac{1}{2^{q/2}} \prod_{\mu= 1}^{q} \Big(x_{\alpha_\mu} - \frac{1}{2\pi}\frac{\partial}{\partial x_{\alpha_\mu}}\Big) \varphi_0
			= \frac{\varphi_0}{(4\pi)^{q/2}} \prod_{j=1}^p H_{q_j}(\sqrt{2\pi}x_{j}),
			\ee
			where $H_n$ denotes the $n$-th Hermite polynomial 
			and $\|\overline{\alpha}\|_1 = \sum_j^p q_j = q$. 
		\end{rem}
		
		\begin{defi}\label{def;QgenandPgen}
			We define polynomials on~$\RR^{p,q}$ in the variables $\underline{x} = (x_1,\hdots, x_{p+q})$ as
			\ba\label{eq:defGpolpol}
			\Gpol_{\aalpha}\big(\underline{x}\big) 
			&\coloneqq \Gpol^{\overline{\alpha}}(\underline{x}) 
			\coloneqq \frac{1}{(4\pi)^{q/2}}\prod_{j=1}^{p} H_{q_j}\big(\sqrt{2\pi}x_{j}\big), 
			\\
			\pol_{\aalpha}\big(\underline{x}\big)
			&\coloneqq
			\pol^{\overline{\alpha}}(\underline{x})
			\coloneqq
			2^{q/2}\prod_{j=1}^p x_{j}^{q_j}.
			\ea
			We denote with the same symbols also the polynomials on~$V$ defined in the same way using the coordinates~$x_1,\dots,x_{p+q}$ with respect to the chosen basis~$(\basevec_j)_j$.
		\end{defi}
		The usage of counting vectors~$\aalphau$ is essential for investigating the injectivity of the Funke--Millson twist of the Kudla--Millson lift, as we will show in Section~\ref{sec;caseofsymmpowers}. 
		We define the linearly independent vectors~$\omega_{\aalpha} \coloneqq \bigwedge_{\mu= 1}^{q} \omega_{\alpha_\mu, \mu}$, where~$\aalpha$ is a multi-index as above, in order to lighten the notation.
		
		Summarizing, we may rewrite~$\varphi_{\text{\rm KM}}\in\big[\mathcal{S}(V)\otimes{\bigwedge}^q(\mathfrak{p}^*)\big]^K$ over the base point~$z_0\in\hermdom$ as
		\ba\label{eq;finphiKMb2}
		\varphi_{\text{\rm KM}}(\genvec,z_0)
		=\sum_{\aalpha} \Big(\Gpol_{\aalpha}\varphi_0\Big)(\genvec) \cdot \omega_{\aalpha},
		\ea
		for every $\genvec\in V$.
		This formula provides an explicit reformulation of what is written in~\cite[p.~$65$]{brfutwo}.
		A notable example is when~${q_i \in \{0,1\}}$ for all~$1 \leq i \leq p$, in which case we have~$\Gpol_{\aalpha} (v) = \pol_{\aalpha}(v)$.

	\begin{rem}\label{rem;finformphiKMglobal}
	In~\eqref{eq;finphiKMb2} we consider~$\varphi_{\text{\rm KM}}$ as a~$K$-invariant function in~${\mathcal{S}(V)\otimes{\bigwedge}^q(\mathfrak{p}^*)}$.
	To construct a global $G$-invariant function in~$\mathcal{S}(V)\otimes \mathcal{A}^q(\hermdom)$, we extend~\eqref{eq;finphiKMb2} to the whole~$\hermdom$ by means of~\eqref{eq;isoschwartzbigwedge}.
	In fact, if for every~$z\in \hermdom$ we choose an isometry~$g\in G$ such that~$g\colon z\mapsto z_0$, then we have that
	\be\label{eq;finformglobspreadphikm}
	\varphi_{\text{\rm KM}}(\genvec,z)=
	g^*\varphi_{\text{\rm KM}}\big(g(\genvec),z_0\big)
	=\sum_{\aalpha} \Big(\Gpol_{\aalpha}\varphi_0\Big)\big(g(\genvec)\big) \cdot \, g^\ast ( \omega_{\aalpha} ).
	\ee
	Since the function is $K$-invariant at the base point $z_0$ the value above does not depend on the choice of $g\colon z \mapsto z_0$. 
	\end{rem}

	\section{The Kudla--Millson theta function}\label{sec;KMthetagen1}	
		In this section we illustrate how to rewrite the Kudla--Millson theta function in terms of Siegel theta functions associated to homogeneous polynomials of degree~$q$. 
	
	As in the previous sections, we denote by~$\big(L,(\argdot {,} \argdot)\big)$ an indefinite even lattice of signature~$(p,q)$.
	We fix once and for all an integer $k=(p+q)/2$ and an orthogonal basis~$(\basevec_j)_j$ of $V=L\otimes\RR$ such that $\basevec_j^2=1$, for every $j=1,\dots,p$, and $\basevec_{p+i}^2=-1$ for $i = 1 , \hdots, q$.
	The choice of such a basis is equivalent to the choice of an isometry $g_0\colon V\to \RR^{p,q}$.
	
	\subsection{Vector-valued modular forms}\label{sec;backgroundvvmodforms}
	We provide a brief overview of elliptic modular forms, vector-valued with respect to the Weil representation associated to the lattice~$L$.
	The reader may find further details in~\cite[Section~$1$]{br;borchp}.
	\\

	Let~$\CC[L'/L]$ be the group algebra arising from the discriminant group of~$L$.
	We denote the standard basis of~$\CC[L'/L]$ by~$(\mathfrak{e_\borgamma})_{\borgamma\in L'/L}$ and write~$\langle\argdot {,} \argdot \rangle$ for the standard scalar product on~$\CC[L'/L]$ that is $\CC$-antilinear in the second component
	
	The Weil representation~$\dwrep{L}$ is a unitary representation of~$\Mp_2(\ZZ)$ on~$\CC[L'/L]$.
	It is defined on the standard generators 
	\[
	T=\left(\big(\begin{smallmatrix}
	1 & 1\\ 0 & 1
	\end{smallmatrix}\big) , 1\right)
	\qquad\text{and}\qquad
	S=\left(\big(\begin{smallmatrix}
	0 & -1 \\ 1 & 0
	\end{smallmatrix}\big),\sqrt{\tau}\right)
	\]
	of the metaplectic group as 
	\bas
	\dwrep{L}(T)\mathfrak{e}_\borgamma = e\big(q(\borgamma)\big)\mathfrak{e}_\borgamma
	\qquad
	\text{and}
	\qquad
	\dwrep{L}(S)\mathfrak{e}_\borgamma = \frac{i^{(p-q)/2}}{|L'/L|^{1/2}}
	\sum_{\delta\in L'/L}
	e\big(-(\borgamma,\delta)\big)\mathfrak{e}_\delta.
	\eas
	
	Let~$k\in\frac{1}{2}\ZZ$.
	A modular form of weight~$k$ with respect to~$\dwrep{L}$ and~$\Mp_2(\ZZ)$ is a holomorphic function~$f\colon\HH\to\CC[L'/L]$ that is holomorphic at the cusp~$\infty$ and satisfies
	\[
	f(\gamma\cdot \tau)
	=
	\phi(\tau)^{2k} \dwrep{L}(\gamma,\phi)f(\tau)\qquad\text{for every~$(\gamma,\phi)\in\Mp_2(\ZZ)$.}
	\]
	The condition that~$f$ is holomorphic at~$\infty$ means that its Fourier expansion is of the form
	\[
	f(\tau)
	=
	\sum_{\borgamma\in L'/L}
	\sum_{\substack{n\in\ZZ+q(\borgamma)\\ n\ge0}}
	c(f_\borgamma,n) e(n \tau) \mathfrak{e}_\borgamma,
	\]
	where~$f_\borgamma$ is the~$\borgamma$-component of the vector-valued function~$f$.
	
	\subsection{Locally symmetric spaces}
	Let~$\mathcal{D}$ be the symmetric domain arising from~$G=\SO(V)$.
	This is a real analytic manifold of dimension~$pq$, and may be identified with the quotient~$G/K$, for some compact maximal subgroup~$K$ of~$G$.
	We realize~$\mathcal{D}$ as the Grassmannian of negative-definite subspaces of dimension~$q$ in~$V$.
	
	Let~$\Gamma\subset\SO(L)$ be a torsion-free subgroup of finite index that acts trivially on the discriminant~$L'/L$.
	The quotient~$X=\Gamma\backslash\mathcal{D}$ inherits a structure of real analytic manifold from~$\mathcal{D}$.
	We refer to it as the \emph{locally symmetric space arising from~$L$}.
	
	The symmetric space~$\mathcal{D}$ is Hermitian if and only if~$p=2$ or~$q=2$.
	In this case, the complex structure of~$\mathcal{D}$ induces a structure of a complex manifold on the quotient space~$X$.
	The latter is a quasi-projective complex variety by the Theorem of Baily and Borel.
	It can be projective only if~$\dim X<3$.

		\subsection{Fundamentals on the Kudla--Millson theta function.}\label{sec;FundamonKMliftgen1}
	We provide here a brief introduction on the Kudla--Millson theta function.
	
	\begin{defi}
	The Kudla--Millson theta function is defined as
	\be\label{eq;thetaserKMdef}
	\Theta(\tau,z,\varphi_{\text{\rm KM}})
	= y^{-k/2} \sum_{\borgamma \in L'/L} \sum_{\lambda\in L + \borgamma}\Big(\wrep(g_\tau)\varphi_{\text{\rm KM}}\Big)(\lambda,z)
	\mathfrak{e}_\borgamma
	\ee
	for every~$\tau=x+iy\in\HH$ and~$z\in\Gr(V)$, where~$g_\tau=\big(\begin{smallmatrix}
	1 & x \\ 0 & 1\end{smallmatrix}\big)\Big(\begin{smallmatrix}
	y^{1/2} & 0\\
	0 & y^{-1/2}
	\end{smallmatrix}\Big)$ is the standard element of~$\SL_2(\RR)$ mapping~$i$ to~$\tau$, and~$\wrep$ is the Schrödinger model of the Weil representation.
	\end{defi}
	
	In the variable $\tau \in \HH$, the function~$\Theta(\tau,z,\varphi_{\text{\rm KM}})$ is a non-holomorphic modular form of weight~$k$ with respect to~$\dwrep{L}$.
	It is a closed $q$-form with respect to the variable~$z\in\Gr(V)$.
	If we denote by~$A^k_{L}$ the space of \textcolor{\myblue}{real} analytic functions from~$\HH$ to~$\CC[L'/L]$ satisfying the weight~$k$ modular transformation property with respect to the Weil representation~$\dwrep{L}$ and~$\Mp_2(\ZZ)$, we may write~${\Theta(\tau,z,\varphi_{\text{\rm KM}})\in A^k_{L}\otimes\mathcal{Z}^q(\hermdom)}$.
	
	In fact, the Kudla--Millson theta function is~$\Gamma$-invariant for every finite index subgroup~$\Gamma$ in~$\SO(L)$ preserving~$L'/L$.
	This implies that the theta function descends to an element of~${A^k_{L}\otimes\mathcal{Z}^q(X)}$ for every locally symmetric space~$X$ arising from~$L$.
	
	Kudla and Millson have shown in~\cite{kumi;intnum} that the cohomology class~$[\Theta(\tau,z,\varphi_{\text{\rm KM}})]$ is a \emph{holomorphic} modular form of weight~$k$ with values in~$H^{q}(X,\CC)$, and coincides with Kudla's generating series of special cycles~\cite[Theorem~3.1]{Kudla;speccycl}.

	Using the spread~\eqref{eq;finformglobspreadphikm} of~$\varphi_{\text{\rm KM}}$, we may rewrite the Kudla--Millson theta function as
	\ba\label{eq;KMdeg1recallexpl}
	\Theta(\tau,z,\varphi_{\text{\rm KM}})
	= \sum_{\aalpha} \underbrace{y^{-k/2} \sum_{\borgamma \in L'/L} \sum_{\lambda\in L + \borgamma} \Big(\wrep(g_\tau)(\Gpol_{\aalpha}\varphi_0)\Big)\big(g(\lambda)\big) \basee_\borgamma}_{\eqqcolon \FFa(\tau,g)}
	\otimes 
	\, g^*( \omega_{\aalpha} ),
	\ea
	where $g\in G$ is any isometry of $V=L\otimes\RR$ mapping $z$ to $z_0$, and~$\Gpol_{\aalpha}$ is the polynomial constructed in Definition~\ref{def;QgenandPgen}.
	Since the Kudla--Millson Schwartz function $\varphi_{\text{\rm KM}}$ is the spread to the whole~$\hermdom=\Gr(V)$ of an element of~$\mathcal{S}(V)\otimes{\bigwedge}^q(\mathfrak{p}^*)$ which is~$K$-invariant, the definition of~$\Theta(\tau,z,\varphi_{\text{\rm KM}})$ does not depend on the choice of~$g$ mapping~$z$ to~$z_0$.
	
	As we will see in Section~\ref{sec;follBorcherds}, it is possible to rewrite the auxiliary functions~$\FFa(\tau,g)$ arising as in~\eqref{eq;KMdeg1recallexpl} in terms of the Siegel theta functions introduced by Borcherds in~\cite{bo;grass}.

	
	\subsection{Siegel theta functions}\label{sec;Siegthetafu}
	Let~$L$ be an indefinite even lattice of general signature~$(p,q)$.
	In this section we recall how to construct Siegel theta functions attached to~$L$, as in~\cite[Section~$4$]{bo;grass}.
	This is ancillary to Section~\ref{sec;follBorcherds}, where we will rewrite the Kudla--Millson theta function associated to~$L$ in terms of Siegel theta functions. 
	
	Recall that we chose the base point~$z_0$ of~$\Gr(V)$ defined as~$z_0=g_0^{-1}(\RR^{0,q})$.
	For a given subspace $z \in \Gr(V)$ and a vector $v \in L \otimes \RR$ we denote by $v_{z}$ and $v_{z^\perp}$ 
	the orthogonal projections of $v$ to $z$ and $z^\perp$. 
	
	The Laplacian~$\Delta$ on~$\RR^{p,q}$ and its exponential are the operators defined respectively as
	\bes
	\Delta=\sum_j\frac{\partial^2}{\partial x_j^2}\qquad\text{and}\qquad\exp\Big(-\frac{\Delta}{8\pi y}\Big)=\sum_{m=0}^\infty\frac{1}{m!}\Big(-\frac{\Delta}{8\pi y}\Big)^m.
	\ees
	\begin{defi}\label{def;deg1genSiegth}
		Let~$\pol$ be a homogeneous polynomial on $\RR^{p,q}$ of degree $(m^+,m^-)$, i.e.\ homogeneous of degree $m^+$ in the first~$p$ variables, and homogeneous of degree~$m^-$ in the last~$q$ variables.
		For simplicity we write~$\pol(\genvec)$ for the value of~$\pol$ at~$g_0(\genvec)$, where~$\genvec\in V$.
		This means that we consider~$\pol$ as a polynomial in the coordinates of~$\genvec$ with respect to the chosen basis~$(\basevec_j)_j$ of~$V$.
		The Siegel theta component $\theta_{\borgamma}$ associated to~$L$, $\borgamma \in L'/L$ and~$\pol$ is defined as
		\ba
		\theta_{\borgamma}(\tau,\boralpha,\borbeta,g,\pol)
		&= y^{q/2 +m^-}\sum_{\lambda\in L + \borgamma}\exp(-\Delta /8\pi y)(\pol)\big(g(\lambda+\borbeta)\big)\\
		&\quad \times e\Big(\tau q\big((\lambda+\borbeta)_{z^\perp}\big)+\bar{\tau}q\big((\lambda+\borbeta)_z\big)-(\lambda+\borbeta/2,\boralpha)\Big),
		\ea
		for every $\tau=x+iy\in\HH$, $\boralpha,\borbeta\in L\otimes\RR$, 
		and~$g\in G$, where~$z=g^{-1}(z_0)\in\Gr(V)$.
		Further, the Siegel theta function~$\Theta_{L}$ is defined as 
		\ba\label{eq;bormegagentheta}
		\Theta_L (\tau,\boralpha,\borbeta,g,\pol) = \sum_{ \borgamma \in L'/L} \theta_{\borgamma}(\tau,\boralpha,\borbeta,g,\pol) \mathfrak{e}_{\borgamma}.
		\ea
		
		If~$\delta=\nu=0$, then we omit these in the notation and write~$\Theta_{L}(\tau,g,\pol)$.
	\end{defi}
	We remark that the factor~$y^{q/2 + m^-}$ in Definition~\ref{def;deg1genSiegth} does not appear in Borcherd's work.
	That factor
	will lighten the notation in the subsequent sections.
	\begin{rem}\label{rem;harmpol}
		If the polynomial~$\pol$ is \emph{harmonic}, i.e.~$\Delta\pol=0$, then~$\exp\big(-\Delta/8\pi y\big)(\pol)=\pol$.
		This is the case for the polynomials~$\pol_{\aalpha}$ presented in Definition~\ref{def;QgenandPgen} if~$\alpha_i\neq\alpha_j$ for all $i \neq j$.
		Otherwise, the polynomials~$\pol_{\aalpha}$ are homogeneous but non-harmonic.
	\end{rem}
	
	The following modular transformation formula is provided by~\cite[Theorem~4.1]{bo;grass}.
	\begin{thm}[Borcherds]\label{thm;borchmodtransftheta4}
		Let $L$ be a lattice of signature $(p,q)$.
		If~$\pol$ is a homogeneous polynomial of degree $(m^+,m^-)$ on~$\RR^{p,q}$, then
		\bas
		\Theta_L(\gamma\cdot\tau, a\boralpha+b\borbeta,c\boralpha+d\borbeta,g,\pol)=(c\tau+d)^{(p-q)/2+(m^+-m^-)} \rho_L(\gamma) \Theta_L(\tau,\boralpha,\borbeta,g,\pol),
		\eas
		for every $\gamma=\left(\left(\begin{smallmatrix}
			a & b\\ c & d
		\end{smallmatrix}\right) , \sqrt{c \tau + d} \right) \in \Mp_2(\ZZ)$.
	\end{thm}
	

	\subsection{Reduction to smaller lattices}\label{sec;redsmallerlat}
	
	The following results illustrates how to decompose the Siegel theta function attached to~$L$ with respect to certain sublattices of $L$. 
	For that purpose we need to introduce some notation. 
	Let $u$ denote a primitive norm $0$ vector of $L$, and let $u' \in L'$ satisfy $(u,u')=1$.
	Define $\brK \coloneqq (L \cap u^\perp)/\ZZ u$ and write $N$ for the smallest positive number in $(u,L)$, so that $|L'/L| = N^2 |\brK'/\brK|$. 
	Further, let $L_0'$ denote the sublattice of $L'$ defined as
	\[
	L_0' \coloneqq \{ \lambda \in L' : (\lambda,\genU) \equiv 0 \; \text{ mod } N\}.
	\]
	We denote by~$\prol : L_0' \to \brK'$ the projection constructed in \cite[(2.7)]{br;borchp}, therein called $p$. 
	This map is such that $\prol(L) = \brK$, and it induces a surjective map $L_0'/L \to \brK'/\brK$ which will be also denoted by $\prol$. 
	Further, recall that $L_0'/L = \{ \lambda \in L'/L : (\lambda,\genU) \equiv 0 \mod n\}$. 
	Note that a hyperbolic (orthogonal) split $L = \brK \oplus U$ with~$\genU$ and~$\genUU$ as standard basis for the hyperbolic plane~$U$ represents a special case in this setting that is revisited in Section~\ref{sec;injKMgenus1}. 
	
	\begin{defi}\label{defi;not&deffromborw}
		Let~$z\in \Gr(V)$, and let~$g\in G$ be such that \textcolor{\myblue}{$g z = z_0$}. 
		\textcolor{\myblue}{W}e denote by~$w$ the orthogonal complement of $\genU_z$ in~$z$, and by~$w^\perp$ the orthogonal complement of~$\genU_{z^\perp}$ in~$z^\perp$.
		We denote by~$\borw \colon V\to V$ the linear map defined as~$\borw(\genvec)=g(\genvec_{w^\perp}+\genvec_w)$.
	\end{defi}
	The linear map $\borw$ is an isometry from $w^\perp\oplus w$ to its image and vanishes on~$\RR\genU_{z^\perp}\oplus\RR\genU_z$.

	\begin{defi}\label{def;bordefimplpol}
		Let $z\in \Gr(V)$, and let $g\in G$ be such that $g$ maps $z$ to $z_0$.
		For every homogeneous polynomial $\pol$ of degree $(m^+,m^-)$ on~$\RR^{p,q}$, we define the homogeneous polynomials~$\pol_{\borw,h^+,h^-}$, of degrees respectively $(m^+-h^+,m^--h^-)$ on~${\borw(V)\cong\RR^{p-1,q-1}}$, by
		\be\label{eq;bordefimplpol}
		\pol\big(g(\genvec)\big)
		=
		\sum_{h^+,h^-}(\genvec,\genU_{z^\perp})^{h^+}\cdot(\genvec,\genU_z)^{h^-}\cdot\polw{\borw}{h^+}{h^-}\big(\borw(\genvec)\big).
		\ee
	\end{defi}

	\begin{rem}\label{ref;howtorewrourpolgtilde}
		The polynomials~$\pol_{\aalpha}$ are homogeneous of degree $(q,0)$, hence we may simplify~\eqref{eq;bordefimplpol} to
		\be\label{eq;decpolborh-0}
		\pol_{\aalpha}\big(g(\genvec)\big)
		=
		\sum_{h^+=\,0}^q(\genvec,\genU_{z^\perp})^{h^+}\cdot\polw{\aalpha,\borw}{h^+}{0}\big(\borw(\genvec)\big).
		\ee
	\end{rem}
	
	The following result provides a formula to compute~$\polw{\aalpha,\borw}{h^+}{0}$.
	In order to state it properly, 
	recall from Remark \ref{rem;KMsFunctionpart} the construction of the counting vector ${\overline{\alpha} = (q_1, \hdots, q_p) \in \ZZ_{\geq 0}^p}$ associated to $\aalpha$. 
	For $k = (k_1,\dots,k_p) \in \ZZ_{\geq0}^p$ we introduce the notation
	\[
		\binom{\overline{\alpha}}{k} \coloneqq \prod_{j = 1}^p \binom{q_j}{k_j}.
	\]
	
	\begin{lemma}\label{lemma;rewritborchimpldecourpol}
		For every~$z\in\Gr(V)$ and~${g\in G}$ such that $g$ maps $z$ to $z_0$, the polynomial~$\polw{\aalpha,\borw}{h^+}{0}$ arising from the decomposition~\eqref{eq;decpolborh-0} of~$\pol_{\aalpha}$ may be computed as
		\ba\label{eq;closefforborpor}
		\polw{\aalpha,\borw}{h^+}{0}\big(\borw(\genvec)\big)
		= \frac{2^{q/2}}{ u_{z^\perp}^{2h^+}} \sum_{\substack{k = (k_1, \hdots, k_p) \\ \textcolor{\myblue}{\|k\|_1 = h^+} }} \binom{\overline{\alpha}}{k}   \prod_{j}^p \big(g(u),\basevec_{j}\big)^{k_j} (\borw(\genvec),\basevec_{j})^{q_j-k_j}.
		\ea
	\end{lemma}
	\begin{proof}
		For every $\genvec\in L\otimes\RR$, we denote by $x_j$ the coordinate of $\genvec$ with respect to the standard basis $\basevec_1,\dots,\basevec_{p+q}$ of $L\otimes\RR$.
		We recall that
		\bes
		\pol_{\aalpha}(\genvec) = 2^{q/2}\prod_{\mu=1}^q x_{\alpha_\mu} = 2^{q/2} \prod_{j=1}^p (\genvec,\basevec_{j})^{q_j}.
		\ees
		If $g\in G=\SO(L\otimes\RR)$, then $\pol_{\aalpha}\big( g(\genvec)\big)=2^{q/2} \prod_{j = 1}^p \big(\genvec,g^{-1}(\basevec_{j})\big)^{q_j}$.
		To rewrite the latter polynomial as in~\eqref{eq;decpolborh-0}, we rewrite $\big(\genvec,g^{-1}(\basevec_{j})\big)$ in terms of $(\genvec,\genU_{z^\perp})$.
		
	Since~$z=g^{-1}(z_0)$, the vectors~$g^{-1}(\basevec_{j})$ lie in~$z^\perp$.
		Recall that $w$ (resp.\ $w^\perp$) is the orthogonal complement of $u_z$ (resp.\ $u_{z^\perp}$) in $z$ (resp.\ $z^\perp$).
		We may decompose
		\be\label{eq;replprimstdbas}
		g^{-1}(\basevec_{j})=s_{j} u_{z^\perp} + v_{j}',\qquad\text{for $1 \leq j \leq p$,}
		\ee
		for some $s_{j}\in\RR$, where $v_{j}'$ is the orthogonal projection of $g^{-1}(\basevec_{j})$ to $w^\perp$.
		
		We use~\eqref{eq;replprimstdbas} to rewrite~$\pol_{\aalpha}\big( g(\genvec)\big)$ as
		\ba\label{eq;decompolabourcase}
		\pol_{\aalpha}\big( g(\genvec)\big)	
		&=	2^{q/2} \prod_{j=1}^p \big( v , s_{j} u_{z^\perp} + v'_{j} \big)^{q_j} \\
		&=	2^{q/2} \prod_{j=1}^p \sum_{k_j=0}^{q_j} \binom{q_j}{k_j} s_{j}^{k_j} (v,u_{z^\perp})^{k_j} (v,v'_{j})^{q_j-k_j} \\
		&=	2^{q/2} \sum_{h^+ =\, 0}^{q} (v,u_{z^\perp})^{h^+} 
		\sum_{\substack{(k_1,\hdots, k_p) \in \ZZ_{\geq0}^p \\ \textcolor{\myblue}{\|k\|_1 = h^+}}} \binom{(q_1,\dots,q_p)}{(k_1,\dots,k_p)}   \prod_{j}^p s_{j}^{k_j} (v,v'_{j})^{q_j-k_j}. \\
		\ea
		Here we have grouped the terms with respect to the sum of the values of $k_j$. 
		Comparing~\eqref{eq;decompolabourcase} with~\eqref{eq;decpolborh-0}, we deduce that
		\bes
		\polw{\aalpha,\borw}{h^+}{0}\big(\borw(\genvec)\big)= 2^{q/2} 
		\sum_{\substack{k \in \ZZ_{\geq0}^p \\ \textcolor{\myblue}{\|k\|_1 = h^+}}} \binom{\overline{\alpha}}{k}   \prod_{j}^p s_{j}^{k_j} (v,v'_{j})^{q_j-k_j}.
		\ees
		Since $u_{z^\perp}$ is orthogonal to $w^\perp$ by construction, it follows that
		\bes
		s_{j}=\frac{\big(\genU_{z^\perp},g^{-1}(\basevec_{j})\big)}{\genU_{z^\perp}^2}=
		\frac{\big(g(\genU),\basevec_{j}\big)}{\genU_{z^\perp}^2}.
		\ees
		Moreover, since~$\basevec_j$ is orthogonal to~$g(\genvec_w)$ for every~$j\le p$, we may rewrite
		\bes
		(\genvec,\genvec_j')=\big(\genvec_{w^\perp},g^{-1}(\basevec_j)\big)=\big(\borw(\genvec),\basevec_j\big).\qedhere
		\ees
	\end{proof}

	Borcherds expresses the theta function $\theta_{\borgamma}$ in terms of Siegel theta functions associated to~$K$; see~\cite[Theorem~5.2]{bo;grass}.
	Kiefer rewrote this result to mimic a Poincaré series which is advantageous for applying unfolding; see {\cite[Thm 6.4]{Kiefer2022}}. 
	The following is Kiefer's result rewritten with respect to our setting.

	\begin{thm}[Borcherds]\label{thm;fromborchspltheta} 
		Let $L$ be an even lattice of signature~$(p,q)$, let~$\mu \in L \otimes \RR$ be the vector defined as
		\bes
		\mu=-\genUU+\genU_{z^\perp}/2\genU_{z^\perp}^2+\genU_z/2\genU_z^2, 
		\ees
		and write $\lvert\genU_{z^\perp}\rvert \coloneqq \sqrt{\genU_{z^\perp}^2}$, as well as $\overline{\Gamma}_\infty$ for the stabilizer of $\infty$ in $\Mp_2(\ZZ)$. 
		Then 
		\bas
		\,&\Theta_{L}(\tau, g,\pol_{\aalpha})
		\\
		&\quad=\frac{1}{\sqrt{2} \lvert\genU_{z^\perp}\rvert} \Theta_{\borK}(\tau,\borw,\polw{\aalpha,\borw}{0}{0}) \sum_{l \in \ZZ/N\ZZ} \mathfrak{e}_{lu/N} \\ 
		&\qquad+ \frac{1}{\sqrt{2}\lvert\genU_{z^\perp}\rvert} \sum_{(\gamma, \phi) \in \overline{\Gamma}_\infty\backslash \Mp_2(\ZZ)} \sum_{h^+=0}^q (-2i)^{-h^+} 
		\sum_{r = 1}^{\infty} r^{h^+} \frac{\phi(\tau)^{-(p+q)}}{\Im(\gamma \tau)^{h^+}}
		\\
		&\qquad\times
		\exp \bigg( - \frac{\pi r^2}{2 \Im(\gamma \tau)\genU_{z^\perp}^2 }\bigg)
		\rho_L(\gamma)^{-1} \bigg[\Theta_{\brK}\big(\gamma \tau,r\mu,0,\borw,\polw{\aalpha,\borw}{h^+}{0}\big) \sum_{l \in \ZZ/N\ZZ} \mathfrak{e}_{\tfrac{lu}{N}} \bigg(-\frac{lr}{N}\bigg) \bigg].
		\eas
		
	\end{thm}	
	
	When we use $\Theta_{\brK}$ in Theorem~\ref{thm;fromborchspltheta}, we should write as argument the orthogonal projection of~$\mu$ to~$\brK\otimes\RR$ instead of~$\mu$.
	However, we use here the same abuse of notation as in~\cite{bo;grass}.

	
	\subsection{The Kudla--Millson theta function in terms of Siegel theta functions.}\label{sec;follBorcherds}
	In this section, we rewrite the Kudla--Millson theta function~$\Theta(\tau,z,\varphi_{\text{\rm KM}})$ in terms of Siegel theta functions.
	We then recall how to rewrite the latter with respect to sublattices of signature $(p-1,q-1)$.

	\begin{lemma}\label{lemma;auxforFaalp}
		Let $y \in \RR_{>0}$. Then for every $ v \in V$ we have
		\ba\label{eq;lemma;auxforFaalp}
		y^{-q/2}\Gpol_{\aalpha}\big(\sqrt{y}\cdot v\big)
		=
		\exp\big(-\Delta/8\pi y\big)(\pol_{\aalpha})(v).
		\ea
	\end{lemma}
	\begin{proof}
		Recall the polynomials~$\Gpol_{\aalpha}$ and~$\pol_{\aalpha}$ from Definition~\ref{def;QgenandPgen}, and that we write~$\genvec$ over the chosen basis of~$V$ as~$v= \sum_{i} x_i \basevec$.
		We rewrite the left-hand side of~\eqref{eq;lemma;auxforFaalp} as
		\bas
		y^{-q/2}\Gpol_{\aalpha}\big(\sqrt{y}\cdot v\big)=
		\prod_{j=1}^p \left(\frac{1}{4\pi y}\right)^{q_j/2} H_{q_j}\big(\sqrt{2\pi y} \cdot x_{j}\big),
		\eas
		and the polynomial~$\pol_{\aalpha}$ as
		\bes
		\pol_{\aalpha}(v) = \prod_{j=1}^p 2^{q_j/2}x_{j}^{q_j}.
		\ees
		
		If~$f_1$ and~$f_2$ are smooth functions defined on~$\RR^{p,q}$ such that the variables~$x_j$ defining~$f_1$ are pairwise different with respect to the ones used to define~$f_2$, then~$\Delta (f_1f_2)=\Delta(f_1)f_2 + f_1\Delta(f_2)$.
		This implies that~$\exp(\Delta)(f_1f_2)=\exp(\Delta)(f_1)\cdot \exp(\Delta)(f_2)$.
		We may then split both sides of~\eqref{eq;lemma;auxforFaalp} in factors which depend on pairwise different variables as
		\be\label{eq;proofofmagicsimpl}
		\prod_{j=1}^p \left(\frac{1}{4\pi y}\right)^{q_j/2} H_{q_j}(\sqrt{2\pi y} x_{j})
		=
		\prod_{j=1}^p \exp\big(-\Delta/8\pi y\big)\big(2^{q_j/2}x_{j}^{q_j}\big).
		\ee
		We prove~\eqref{eq;lemma;auxforFaalp} by showing that the~$j$-th factor on the left-hand side of~\eqref{eq;proofofmagicsimpl} equals the~$j$-th factor on the right-hand side of~\eqref{eq;proofofmagicsimpl}.
		We may then assume without loss of generality that~$p=1$ and~$q_1=q$.
		From now on, we omit the indices~$j$ from the notation, assuming that the multi-index~$\aalpha$ has~$q$ entries equal to~$1$.
		
		By evaluating~\eqref{eq;Hermitepolynomialexplicit} at $\sqrt{2 \pi} x$ we have
		\ba\label{eq;thingstomatchmagsympl}
		y^{-q/2}\Gpol_{\aalpha}\big(\sqrt{y} \cdot v\big)
		&=
		q!\sum_{n=0}^{\lfloor q/2\rfloor}\frac{(-1)^n 2^{q/2 -3n}}{n! (q-2n)! \pi^n y^n}x^{q-2n}.
		\ea
		We want to show that this equals~$\exp\big(-\Delta/8\pi y\big)(\pol_{\aalpha})\big( g(v)\big)$.
		Since the polynomial~$\pol_{\aalpha}$ is now in only one variable~$x$, we may replace the Laplacian~$\Delta$ with~$\partial^2/\partial x^2$.
		We deduce that
		\ba\label{eq;finalmatchingmagsympl}
		\exp\big(-\Delta/8\pi y\big)(\pol_{\aalpha})(v)
		&=
		\sum_{n=0}^\infty \frac{(-1)^n}{n! (8\pi y)^n}\bigg(\frac{\partial^2}{\partial x^2}\bigg)^n (2^{q/2} x^q)
		\\
		&=
		q!\sum_{n=0}^{\lfloor q/2 \rfloor}\frac{(-1)^n 2^{q/2 - 3n}}{n! (q-2n)!\pi^n y^n} x^{q-2n},
		\ea
		where in the last passage we used that
		\bes
		\bigg(\frac{\partial^2}{\partial x^2}\bigg)^n (x^q)
		=
		\begin{cases}
			\frac{q!}{(q-2n)!}x^{q-2n}, & \text{if~$q\ge n$},\\
			0, & \text{otherwise}. 
		\end{cases}
		\qedhere
		\ees
	\end{proof}
	Recall that we denote by~$\FFa$ the auxiliary functions appearing in the rewriting~\eqref{eq;KMdeg1recallexpl} of the Kudla--Millson theta function.
	\begin{prop}\label{lemma;FabwithBorch}
		For every multi-index~$\aalpha$, we may rewrite the auxiliary function~$\FFa$ in terms of Siegel theta functions as
		\be\label{eq;FabwithBorch}
		\FFa(\tau,g)=
		\Theta_L(\tau,g,\pol_{\aalpha}),
		\ee
		for every~$\tau=x+iy\in\HH$ and~$g\in G$, where $\pol_{\aalpha}$ is the polynomial as in \eqref{eq:defGpolpol}. 
	\end{prop}
	\begin{proof}
		Let~$k=(p+q)/2$, and let $g_\tau=\big(\begin{smallmatrix} 1 & x\\ 0 & 1 \end{smallmatrix}\big)\left(\begin{smallmatrix} \sqrt{y} & 0\\ 0 & \sqrt{y}^{-1}\end{smallmatrix}\right)$ be the standard element of~$\SL_2(\RR)$ mapping~$i$ to~$\tau=x+iy$.
		Recall the Schrödinger model~$\wrep$ from Definition~\ref{defi;schrmod}. We may compute
		\ba\label{eq;inizFabwithBorch}
		\wrep(g_\tau)\big(\Gpol_{\aalpha}\varphi_0\big)\big( g(\genvec)\big)
		&=
		y^{k/2}\cdot\wrep\left(\begin{smallmatrix} 1 & x\\ 0 & 1\end{smallmatrix}\right)\big(\Gpol_{\aalpha} \varphi_0\big)\big( g(y^{1/2}\genvec)\big)\\
		&=
		y^{k/2}\cdot e\big(xq(\genvec)\big)
		\cdot\Gpol\big(\sqrt{y} \cdot  g(\genvec)\big)
		\cdot \varphi_0\big( g(y^{1/2}\genvec)\big).
		\ea
		Since~$\varphi_0\big( g(y^{1/2}\genvec)\big)=e^{-\pi y(\genvec,\genvec)_z}$, we may deduce that
		\bas
		e\big(\tau q(\genvec_{z^\perp})+\bar{\tau}q(\genvec_z)\big)=e\big(x q(\genvec)\big)\cdot e^{-\pi y(\genvec,\genvec)_z}=e\big(x q(\genvec)\big)\cdot\varphi_0\big( g(y^{1/2}\genvec)\big),
		\eas
		for every $\tau\in\HH$.
		This, together with~\eqref{eq;inizFabwithBorch}, implies that
		\bes
		\wrep(g_\tau)\big(\Gpol_{\aalpha}\varphi_0\big)\big( g(\genvec)\big)=
		y^{k/2}\cdot
		\Gpol_{\aalpha}\big(\sqrt{y} \cdot  g(\genvec)\big)\cdot
		e\big(\tau q(\genvec_{z^\perp})+\bar{\tau}q(\genvec_z)\big).
		\ees
		We may then rewrite the auxiliary function~$\FFa$ as
		\bes
		\FFa(\tau,g)
		=
		\sum_{\borgamma \in L'/L} \sum_{\lambda \in L + \borgamma} \Gpol_{\aalpha}\big(\sqrt{y} \cdot  g(\lambda)\big)\cdot
		e\big(\tau q(\lambda_{z^\perp})+\bar{\tau}q(\lambda_z)\big) \mathfrak{e}_{\borgamma},
		\ees
		from which we deduce it equals~$\Theta_L(\tau,g,\pol_{\aalpha})$ by Lemma~\ref{lemma;auxforFaalp}.
	\end{proof}

	\section{The unfolding of the Kudla--Millson lift}\label{sec;theunfofKMgen1}
	In this section we unfold 
	the genus~$1$ Kudla--Millson theta lift~${\KMliftbase\colon S^k_L\to\mathcal{Z}^q(X)}$, defined as
	\[
	\KMliftbase(f)
	\coloneqq
	\int_{\SL_2(\ZZ)\backslash\HH}y^k \langle f(\tau) , \Theta(\tau,z,\varphi_{\text{\rm KM}}) \rangle \frac{dx\,dy}{y^2}.
	\]
	We recall that it produces~$\Gamma$-invariant closed~$q$-forms on~$\hermdom$, which descend to closed~$q$-forms on the locally symmetric space~$X=\Gamma\backslash\hermdom$.
	
	By means of~\eqref{eq;KMdeg1recallexpl} and Proposition \ref{lemma;FabwithBorch}, we may rewrite~$\KMliftbase$ more explicitly as
	\ba\label{eq;KMlambdaaO}
	\KMliftbase(f)
	=&\sum_{\aalpha}\Big(
	\int_{\SL_2(\ZZ)\backslash\HH} y^{k} \langle f(\tau), \Theta_L(z,g,\pol_{\aalpha}) \rangle \frac{dx\,dy}{y^2}
	\Big) \cdot g^*( \omega_{\aalpha} ),
	\ea
	for every cusp form~$f\in S^k_L$, and for every~$g\in G$ mapping~$z$ to~$z_0$.
	The value of~$\KMliftbase(f)$ on~$z$ does not depend on the choice of such a~$g$.
	We refer to the integrals appearing as coefficients in~\eqref{eq;KMlambdaaO}, namely
	\be\label{eq;intwwtu}
	\intfunct(g) \coloneqq \int_{\SL_2(\ZZ)\backslash\HH}y^{k} \big\langle f(\tau), \Theta_L(z,g,\pol_{\aalpha}) \big\rangle \frac{dx\,dy}{y^2},
	\ee
	as the \emph{defining integrals} of~$\KMliftbase(f)$.
	We will compute the Fourier expansion of such integrals using the \emph{unfolding trick}.
	
	\subsection{Fourier expansions under invariance by Eichler transformations}\label{sec;FexpLlorinvfuncts}
	In this section we recall Fourier expansions of complex-valued functions on~$G=\SO(V)$ invariant with respect to certain Eichler transformations.
	The Hermitian case in signature~$(p,2)$ has already been treated in~\cite[Section~$4$]{zufunfolding} by means of an identification of~$G$ with a product of a compact maximal subgroup of~$G$ and the tube domain model of~$\Gr(V)$.
	In this section we describe Fourier expansions with respect to Eichler transformations induced by a sublattice of~$L$.
	This point of view allows us to avoid to choose any identification.
	
	Since the results of this sections are well-known to experts, we leave most of the details and proofs to the reader.\\
	
	Let~$L$ be an even indefinite lattice of signature~$(p,q)$.
	Recall the construction of the elements~$\genU,\genUU\in L'$ and the sublattice~$M\subset L$ from Section~\ref{sec;redsmallerlat}.
	For every~$g\in G$, we denote by~$z\in\Gr(V)$ the negative definite plane that maps to the base point~$z_0$ under~$g$.
	We consider
	\[
	\mu = -\genUU + \genU/2\genU^2_{z^\perp} + \genU_z/2\genU^2_z,
	\]
	as a function of either~$z\in\Gr(V)$ or~$g\in G$, depending on the context.
	
	If~$q=2$, translations on the tube domain model of~$\Gr(V)$ are provided by certain isometries in~$G$ called \emph{Eichler transformations}.
	We recall them in \emph{general signature} as follows.
	\begin{defi}
	Let~$\lambda\in\brK\otimes\RR$.
	The Eichler transformation~$E(\genU,\lambda)\in G$ associated to the isotropic vector~$\genU\in U$ is defined as
	\bes
	E(\genU,\lambda)(\genvec)=\genvec - (\genvec,\genU)\lambda + (\genvec,\lambda)\genU - q(\lambda)(\genvec,\genU)\genU,\qquad\text{for every~$\genvec\in L\otimes\RR$.}
	\ees
	\end{defi}
	The following technical result is preparatory to the description of Fourier expansions.
	We recall that we denote by~$w$, resp.~$w^\perp$, the orthogonal complement of~$\genU_z$, resp.~$\genU_{z^\perp}$, in~$z$, resp.~$z^\perp$. Furthermore, if~$g\in G=\SO(V)$, then~$\borw$ is the linear map defined as
	\[
	\borw(\genvec)=g(\genvec_{w^\perp}+\genvec_w),\qquad\text{for every~$\genvec\in L\otimes\RR$.}
	\]
	\begin{prop}\label{prop;auxforFexpnohermitdom}
	Let~$\lambda\in\brK\otimes\RR$, and let~$z\in\Gr(V)$.
	We denote by~$g$ an isometry in~$G$ mapping~$z$ to the base point~$z_0$, and by~$\tilde{z}\in\Gr(V)$ the negative definite \textcolor{\myblue}{space} that maps to~$z$ under~$E(\genU,\lambda)$.
	We also denote by~$\tilde{w}$, resp.~$\tilde{w}^\perp$, the orthogonal complement of~$\genU_{\tilde{z}}$, resp.~$\genU_{\tilde{z}^\perp}$, in~$\tilde{z}$, resp.~$\tilde{z}^\perp$.
	Furthermore, we define~$\tilde{\mu}=-\genUU + \genU_{\tilde{z}^\perp}/2\genU^2_{\tilde{z}^\perp} + \genU_{\tilde{z}}/2\genU^2_{\tilde{z}}$.
	\begin{enumerate}[label=(\roman*)]
	\item The isometry~$E(\genU,\lambda)$ fixes~$\genU$, i.e.~$E(\genU,\lambda)(u)=u$. \label{basicpropEich,ufixed}
	\item We have that~$E(\genU,\lambda)(\tilde{w})=w$ and~$E(\genU,\lambda)(\tilde{w}^\perp)=w^\perp$.\label{basicpropEich,imagewwperp}
	\item The linear map~$\borw$ is such that~$\borw(\genvec)=\big(g\circ E(\genU,\lambda)\big)^{\#}(\genvec)$, for every~$\genvec\in\brK\otimes\RR$.\label{auxforFexpnohermitdom,1}
	\item We have that~$\genvec_{w^\perp}^2=\genvec^2_{\tilde{w}^\perp}$ for every~$\genvec\in\brK\otimes\RR$.\label{auxforFexpnohermitdom,3}
	\item \textcolor{\myblue}{Let~$\lambda'\in\brK\otimes\RR$, then~$e\big((\lambda',\tilde{\mu})\big)=e\big((\lambda',\mu)\big)\cdot e\big((\lambda',\lambda)\big)$.}\label{auxforFexpnohermitdom,4}
	\end{enumerate}
	\end{prop}

	It is well-known that if~$F\colon G\to\CC$ is a smooth function invariant with respect to the Eichler transformations~$E(\genU,\lambda)$, where~$\lambda\in \brK$, then~$F$ admits a Fourier expansion of the form
	\be\label{eq;genshapeFexpinvEichler}
	F(g)=\sum_{\lambda\in \brK'}c_\lambda(g) e\big((\lambda,\mu)\big).
	\ee
	Here the Fourier coefficients~$c_\lambda \colon G\to\CC$ are functions invariant with respect to the Eichler transformations~$E(\genU,\genvec)$ with~$\genvec\in\brK\otimes\RR$.
	
	One can easily check that the defining integrals~$\intfunct$ of the lift~$\KMliftbase(f)$ are~$\brK$-invariant with respect to Eichler transformations, hence they admit a Fourier expansion.
	In Section~\ref{sec;Fexpdefintvv} we will use Proposition~\ref{prop;auxforFexpnohermitdom} to show that the series obtained by applying the unfolding method to~$\intfunct$ is actually a Fourier expansion of the same shape as in~\eqref{eq;genshapeFexpinvEichler}.
	
	\subsection{The Fourier expansion of the defining integrals}
	\label{sec;Fexpdefintvv}
	In order to unfold~$\intfunct$, its integrand has to be rewritten as a Poincaré series. 
	This will be achieved by means of Theorem \ref{thm;fromborchspltheta} and Lemma \ref{lemma;scalarprodtosublattice}.
	To state the latter, we need to introduce some notation.
	
	As in~\cite{bo;grass} we associate to any~$f: \HH \to \CC[L'/L]$ a function 
	\[
	F_{\brK}(\tau; r,t) = \sum_{\borgamma \in K'/K} f_\borgamma(\tau ; r ,t) \otimes \mathfrak{e}_\borgamma,
	\]
	where 
	\[
	f_\borgamma(\tau ; r , t) = \sum_{\substack{\lambda \in L_0'/L \\ \prol(\lambda) = \borgamma}} e\left(- r 	(\lambda,\genUU) - rtq(\genUU)\right) f_{L + \lambda + t \genUU}(\tau).
	\]
	If $f$ is a modular form of weight $k$ with respect to $\rho_L$, $F_{\brK}$ is a modular form to $\brK$ of the same weight with respect to $\rho_M$; see \cite[Thm 2.6]{br;borchp}.

	\begin{lemma}[{See \cite[2.4.8]{KieferThesis}}]\label{lemma;scalarprodtosublattice}
		If~$f: \HH \to \CC[L'/L]$ and $g: \HH \to \CC[\borK'/\borK]$, then
		\[
		\left\langle f(\tau) , g (\tau) \sum_{l \in \ZZ/N\ZZ} \mathfrak{e}_{\tfrac{lu}{N}} \left(-\frac{lr}{N}\right) \right\rangle = \langle F_{\borK}(\tau,-r,0) , g(\tau)\rangle.
		\]
	\end{lemma}

	With such a preparatory result, the desired decomposition of the integrand of~$\intfunct$ may be computed.
	
	\begin{lemma}\label{lemma;IntKernelPoincare}
		Define   
		\bas
		h_{\aalpha}(\tau,g) 
		\coloneqq& \frac{1}{\sqrt{2} \lvert \genU_{z^\perp} \rvert} \sum_{h^+} 
		\sum_{r = 1}^{\infty} \left(\frac{r}{2i}\right)^{h^+} \Im(\tau)^{k-h^+} \exp \left( - \frac{\pi r^2}{2 \Im( \tau)\genU_{z^\perp}^2 }\right) \\
		&\times \left\langle F_{\brK}(\tau , - r , 0) , \Theta_{\brK}\big( \tau,r\mu,0,\borw,\polw{\aalpha,\borw}{h^+}{0}\big) \right\rangle. 
		\eas 
		Then the integral kernel of $\intfunct(g)$ may be rewritten as
		\ba\label{eq;unfoldingtrickform}
		y^{k} \langle f(\tau), \Theta_L(z,g,\pol_{\aalpha}) \rangle 
		&=
		\frac{y^{k}}{\sqrt{2} \lvert \genU_{z^\perp} \rvert} \langle F_{\brK}(\tau; 0,0) , \Theta_{\brK}(\tau,\borw,\polw{\aalpha,\borw}{0}{0}) \rangle
		\\
		&\quad +	\sum_{\gamma=\left(\begin{smallmatrix}
				* & *\\ c & d
			\end{smallmatrix}\right)\in\Gamma_\infty\backslash\SL_2(\ZZ)}\hha(\gamma\cdot\tau,g)
		\ea
		for every~$g\in G$ and~$z\in\Gr(V)$ such that~$g\colon z\mapsto z_0$.
	\end{lemma}

	\begin{proof}
		By inserting the decomposition of $\Theta_L(z,g,\pol_{\aalpha})$ provided by Theorem \ref{thm;fromborchspltheta} in the left hand side of \eqref{eq;unfoldingtrickform} we obtain
	\begin{align*}
		&\Im(\tau)^k \langle f(\tau) , \FFa(\tau,g) \rangle 	\\
		&= \frac{1}{\sqrt{2} \lvert \genU_{z^\perp} \rvert} \Im(\tau)^{k} \bigg\langle f(\tau), \Theta_{\borK}(\tau,\borw,\polw{\aalpha,\borw}{0}{0}) \sum_{l \in \ZZ/N\ZZ} \mathfrak{e}_{\tfrac{lu}{N}} \bigg\rangle  \\ 
		&\quad + \frac{\Im(\tau)^k}{\sqrt{2} \lvert \genU_{z^\perp} \rvert} \sum_{(\gamma,\phi) \in \overline{\Gamma}_\infty\backslash \Mp_2(\ZZ)} \sum_{h^+} 
		\sum_{r = 1}^{\infty} \left(\frac{r}{2i}\right)^{h^+} \frac{ \overline{\phi(\tau)}^{- 2k} }{\Im(\gamma \tau)^{h^+}} \exp \left( - \frac{\pi r^2}{2 \Im(\gamma \tau)\genU_{z^\perp}^2 }\right) \\
		&\quad \times \left\langle f(\tau) , \rho_L(\gamma)^{-1} \bigg(\Theta_{\brK}\big(\gamma \tau,r\mu,0,\borw,\polw{\aalpha,\borw}{h^+}{0}\big) \sum_{l \in \ZZ/N\ZZ} \mathfrak{e}_{\tfrac{lu}{N}} \left(-\frac{lr}{N}\right) \bigg) \right\rangle. \\
		\intertext{By Lemma \ref{lemma;scalarprodtosublattice} and the transformation property of the imaginary part $\Im$ under Möbius transformations the expression above equals the desired Poincaré series:}
		&\quad \frac{1}{\sqrt{2} \lvert \genU_{z^\perp} \rvert} \Im(\tau)^{k} \langle F_{\borK}(\tau,0,0), \Theta_{\borK}(\tau,\borw,\polw{\aalpha,\borw}{0}{0}) \rangle  \\ 
		&\quad + \frac{1}{\sqrt{2} \lvert \genU_{z^\perp} \rvert} \sum_{\gamma \in \Gamma_\infty\backslash \SL_2(\ZZ)} \sum_{h^+}  
		\sum_{r = 1}^{\infty} \left(\frac{r}{2i}\right)^{h^+} \Im(\gamma \tau)^{k-h^+} \exp \left( - \frac{\pi r^2}{2 \Im(\gamma \tau)\genU_{z^\perp}^2 }\right) \\
		&\quad \times \left\langle F_{\brK}(\gamma\tau , - r , 0) , \Theta_{\brK}\big(\gamma \tau,r\mu,0,\borw,\polw{\aalpha,\borw}{h^+}{0}\big) \right\rangle. 
	\end{align*}
	This is exactly the stated relation if $h_{\aalpha}$ is inserted in \eqref{eq;unfoldingtrickform}.
	\end{proof}

	By means of Lemma \ref{lemma;IntKernelPoincare} we may then unfold the defining integrals~\eqref{eq;intwwtu} of the Kudla--Millson lift as
	\ba\label{eq;firstunftrickap}
	\intfunct(g) 
	=&\int_{\SL_2(\ZZ)\backslash\HH}y^k \langle f(\tau) , \FFa(\tau,g) \rangle\frac{dx\,dy}{y^2}\\
	=&\int_{\SL_2(\ZZ)\backslash\HH}\frac{y^{k}}{\sqrt{2} \lvert \genU_{z^\perp} \rvert} \langle F_{\brK}(\tau; 0,0) , \Theta_{\brK}(\tau,\borw,\polw{\aalpha,\borw}{0}{0}) \rangle \frac{dx\,dy}{y^2} \\
	&+ 2\int_{\Gamma_\infty\backslash\HH}\hha(\tau,g)\frac{dx\,dy}{y^2}.
	\ea
	
	Next, 
	we compute a Fourier expansion of the defining integral~\eqref{eq;intwwtu} of~$\KMliftbase$, for every~$\aalpha$.
	To do so, we begin by rewriting the last term of the right-hand side of~\eqref{eq;firstunftrickap} by Lemma~\ref{lemma;IntKernelPoincare} as
	\ba\label{eq;firstunftrickap2}
	2\int_{\Gamma_\infty\backslash\HH}\hha(\tau,g)\frac{dx\,dy}{y^2}=
	\int_{0}^{+\infty}\int_0^1\frac{\sqrt{2}y^{k-2}}{\lvert \genU_{z^\perp} \rvert}\sum_{h^+=0}^{q}(2iy)^{-h^+}\hspace{2cm}\\
	\times
	\sum_{r\ge1} r^{h^+}\cdot \exp\Big(
	-\frac{\pi r^2}{2y\genU_{z^\perp}^2}
	\Big) \left\langle F_{\brK}(\tau,0,0), \Theta_{\brK}\big(\tau,r\mu,0,\borw,\polw{\aalpha,\borw}{h^+}{0}\big) \right\rangle dx\,dy.
	\ea
	We are going to replace in~\eqref{eq;firstunftrickap2} the cusp form~$f$ with its Fourier expansion, and the Siegel theta function~$\Theta_{\brK}$ with its defining series.
	We denote the Fourier expansion of $f$ by
	\be\label{eq;fexpf}
	f(\tau)= \sum_{\lambda \in L'/L} \sum_{n>0}c(f_{\lambda},n)e(n\tau) \cdot \mathfrak{e}_\lambda = \sum_{\lambda \in L'/L} \sum_{n>0}c_n(f_\lambda)\exp(-2\pi ny)e(nx) \cdot \mathfrak{e}_\lambda.
	\ee
	
	Recall that we denote by~$(\argdot {,} \argdot)_w$ the \emph{standard majorant} of~$\brK\otimes\RR$ with respect to~$w\in\Gr(\brK)$, that is~$(\genvec,\genvec)_w=(\genvec_{w^\perp},\genvec_{w^\perp})-(\genvec_w,\genvec_w)$ for every $\genvec\in \brK\otimes\RR$.
	We rewrite the defining series of~$\Theta_{\brK}$ with respect to the decomposition~${\tau=x+iy}$ in real and imaginary part as
	\ba\label{eq;precFexpofTllor}
	\Theta_{\brK}\big(\tau,r\mu,0,\borw,\polw{\aalpha,\borw}{h^+}{0}\big)
	&=
	y^{(q-1)/2}  
	\sum_{\borgamma \in \brK'/\brK} 
	\sum_{\lambda\in M + \borgamma}  \exp(-\Delta /8\pi y)(\pol)\big(\borw(\lambda)\big)\\
	&\quad \times \exp(- \pi y (\lambda,\lambda)_\omega ) \cdot e\Big(x q(\lambda) - r(\lambda,\mu)\Big) \cdot \mathfrak{e}_\borgamma.
	\ea
	We are now ready to prove the main result of this section.
	It provides a Fourier expansion of~$\intfunct$ with the same shape as in~\eqref{eq;genshapeFexpinvEichler}.
	
	\begin{thm}\label{thm;Fourexpofcoef}
		Let~$f\in S^k_1$ be an elliptic cusp form.
		The Fourier expansion of the defining integrals
		\[
			\intfunct\colon G\to\CC, \qquad g \mapsto 
			\int_{\SL_2(\ZZ)\backslash\HH}
			y^k \big\langle f(\tau) , \FFa(\tau,g) \big\rangle\frac{dx\,dy}{y^2}
		\]
		of the Kudla--Millson lift is
		\ba\label{eq;Iaalpha}
		\intfunct(g) 
		=& 
		\int_{\SL_2(\ZZ)\backslash\HH}
		\frac{y^{k}}{\sqrt{2} \lvert \genU_{z^\perp} \rvert} \big\langle F_{\brK}(\tau; 0,0) , \Theta_{\brK}(\tau,\borw,\polw{\aalpha,\borw}{0}{0}) \big\rangle \frac{dx\,dy}{y^2} \\
		&+\sum_{\lambda \in \borK'} \sum_{\substack{t \geq 1 \\ t \mid \lambda}} \frac{\sqrt{2}}{|\genU_{z^\perp}|} 
		\sum_{h^+} \left(\frac{t}{2i}\right)^{h^+} \sum_{\substack{\lambdatwo \in L_0'/L \\ \pi(\lambdatwo) = \lambda/t + M}}  e\Big(t(\lambdatwo,\genUU)\Big) c(f_{\lambdatwo},q(\lambda)/t^2)  \\
		& \times \int_{0}^{\infty} y^{q + (p  -5)/2 - h^+} 
		\exp\left(-\frac{2\pi y \lambda_{w^\perp}^2}{t^2} - \frac{\pi t^2}{2 y\genU_{z^\perp}^2 } \right) 
		\\
		& \times  \exp(-\Delta /8\pi y)(\polw{\aalpha,\borw}{h^+}{0})\big( \borw(\lambda/t)\big) dy \cdot e\Big( (\lambda,\mu) \Big). \\ 
		\ea
		Here we say that an integer $t\ge 1$ divides $\lambda\in\brK'$, in short~$t|\lambda$, if and only if~$\lambda/t$ is still a lattice vector in $\brK'$.  
	\end{thm}

	\begin{proof}[Proof of Theorem~\ref{thm;Fourexpofcoef}]
		We consider the unfolding~\eqref{eq;firstunftrickap} of~$\intfunct$.
		The first summand of the right-hand side of~\eqref{eq;firstunftrickap} provides the first term on the right hand side of~\eqref{eq;Iaalpha}.
		This is the constant term of the Fourier expansion of~$\intfunct$.
		In fact, one can easily check by using Proposition~\ref{prop;auxforFexpnohermitdom} that such a term is~$\brK$-invariant.
		The fact that this is the only term which is not multiplied by any~$e\big((\lambda,\mu)\big)$, with~$\lambda\in\brK'$ non-zero, follows from the rewriting of the integral of~$\hha$ below.
		
		We now begin the computation of the second summand appearing on the right-hand side of~\eqref{eq;firstunftrickap}.
		First of all, we rewrite
		\begin{equation}\label{eq;inproofofFourierTheorem}
			y^{(1-q)/2}
			\left\langle F_{\brK}(\tau, -r, 0) , \Theta_{\brK}(\tau,r\mu,0,\borw,\polw{\aalpha,\borw}{h^+}{0})\right\rangle
		\end{equation}
		with respect to the real and imaginary parts of~$\tau\in\HH$.
		To do so, we replace~$f$ and~$\Theta_{\brK}$ with~\eqref{eq;fexpf} and~\eqref{eq;precFexpofTllor}, respectively, deducing that \eqref{eq;inproofofFourierTheorem} equals
		\begin{align*}
			&	\sum_{\borgamma \in \borK'/\borK} \sum_{\substack{\lambdatwo \in L_0'/L \\ \prol(\lambdatwo) = \borgamma}} e\left(r (\lambdatwo,\genUU)\right) \sum_{\substack{n \in \ZZ + q(\lambdatwo) \\ n > 0}} c(f_{\lambdatwo},n) e(n\tau) \\
			&\quad \times \sum_{\lambda\in \borK + \borgamma}\exp(-\Delta /8\pi y)(\polw{\aalpha,\borw}{h^+}{0})\big( \borw(\lambda)\big)
			\cdot
			\exp(- \pi y (\lambda,\lambda)_\omega ) \cdot e\Big(r(\lambda,\mu) -x q(\lambda) \Big) \\
			&=	\sum_{\borgamma \in \borK'/\borK} \sum_{\lambda\in \borK + \borgamma} \sum_{\substack{\lambdatwo \in L_0'/L \\ \prol(\lambdatwo) = \borgamma}} \sum_{\substack{n \in \ZZ + q(\lambdatwo) \\ n > 0}} e\left(r (\lambdatwo,\genUU)\right) c(f_{\lambdatwo},n) \exp(-2\pi n y) e(nx) \\
			&\quad \times \exp(-\Delta /8\pi y)(\polw{\aalpha,\borw}{h^+}{0})\big( \borw(\lambda)\big)
			\cdot
			\exp(- \pi y (\lambda,\lambda)_\omega ) \cdot e\Big(r(\lambda,\mu) -x q(\lambda) \Big). 
		\end{align*}
		This may be rewritten by using that a set of representatives for $\lambdatwo$ is given by $\borgamma - (\borgamma,\zeta)/N \cdot u + b/N \cdot u$ where $b$ runs through a set of representatives mod $N$. As a consequence, we obtain the identity $q(\lambdatwo) = q(\borgamma) \equiv q(\lambda)$ so that the above expression equals
		\begin{align*}
			&\sum_{\lambda\in \borK'} \sum_{\substack{\lambdatwo \in L_0'/L \\ \prol(\lambdatwo) = \lambda + M}} \sum_{\substack{n \in \ZZ + q(\lambda) \\ n > 0}} e\left(r (\lambdatwo,\genUU)\right) c(f_{\lambdatwo},n) \exp(-2\pi n y - \pi y (\lambda,\lambda)_\omega ) \\
			&\quad \times \exp(-\Delta /8\pi y)(\polw{\aalpha,\borw}{h^+}{0})\big( \borw(\lambda)\big)
			\cdot
			e \Big(x (n - q(\lambda)) \Big)\cdot e\Big(r(\lambda,\mu) \Big).
		\end{align*}
		We insert the previous formula in the defining formula of~$\hha$ provided by Lemma~\ref{lemma;IntKernelPoincare}, deducing that
		\ba\label{eq;killingx}
		&2\int_{\Gamma_\infty\backslash\HH}\hha(\tau,g)\frac{dx\,dy}{y^2} \\
		&\!\! = \, \frac{\sqrt{2}}{|\genU_{z^\perp}|} \sum_{h^+} 
		\sum_{r \geq 1} \left(\frac{r}{2i}\right)^{h^+} \sum_{\lambda\in \borK'} \sum_{\substack{\lambdatwo \in L_0'/L \\ \prol(\lambdatwo) = \lambda + \borK}} \sum_{\substack{n \in \ZZ + q(\lambda) \\ n > 0}} c(f_{\lambdatwo},n)
		\int_{0}^{\infty} y^{q + (p-5)/2 - h^+}
		\\
		& \quad \times \exp\left(-2\pi n y - \pi y (\lambda,\lambda)_\omega - \frac{\pi r^2}{2 y\genU_{z^\perp}^2 } \right) \exp(-\Delta /8\pi y)(\polw{\aalpha,\borw}{h^+}{0})\big( \borw(\lambda)\big) dy\\
		& \quad \times  \int_{0}^{1} e \Big(x (n - q(\lambda)) \Big) dx \cdot e\Big(r \big((\lambda,\mu) + (\lambdatwo,\genUU)\big) \Big). 
		\ea
		The last integral appearing on the right-hand side of~\eqref{eq;killingx} may be computed as
		\bes
		\int_{0}^{1} e \Big(x \big(n - q(\lambda)\big) \Big) dx 
		=\begin{cases}
			1 & \text{if $n=q(\lambda)$,}\\
			0 & \text{otherwise.}
		\end{cases}
		\ees
		This simplifies the sum over $n$ in \eqref{eq;killingx} leaving only the term $n=q(\lambda)$. 
		Since $c_n(f_\lambdatwo) = 0$ for every $n \leq 0$ we deduce that  
		\begin{align}\label{eq;killedx}
		&2\int_{\Gamma_\infty\backslash\HH}\hha(\tau,g)\frac{dx\,dy}{y^2} \\
		&\!\! = \, \frac{\sqrt{2}}{|\genU_{z^\perp}|} \sum_{h^+} \nonumber
		\sum_{r \geq 1} \left(\frac{r}{2i}\right)^{h^+} \sum_{\lambda\in \borK'} \sum_{\substack{\lambdatwo \in L_0'/L \\ \prol(\lambdatwo) = \lambda + \borK}} c(f_{\lambdatwo},q(\lambda))  \\
		&\quad \times \int_{0}^{\infty} y^{q + (p-5)/2 - h^+} \nonumber
		\exp\left(-2\pi y \lambda_{w^\perp}^2 - \frac{\pi r^2}{2 y\genU_{z^\perp}^2 } \right) 
		\exp(-\Delta /8\pi y)(\polw{\aalpha,\borw}{h^+}{0})\big( \borw(\lambda)\big) dy\\
		&\quad \times  e\Big(r \big( (\lambda,\mu) + (\lambdatwo,\genUU) \big) \Big). \nonumber
		\end{align}
		We rewrite~\eqref{eq;killedx} by gathering the terms associated to~$e\big((\lambda,\mu)\big)$, for every~$\lambda$.
		This can be accomplished simply by replacing the sum~$\sum_{r\ge1}$ with~$\sum_{t\ge1,\,t|\lambda}$, and the lattice vector~$\lambda$ with~$\lambda/t$.
		In this way, we obtain that the expression above equals
		\begin{align*}
		&\frac{\sqrt{2}}{|\genU_{z^\perp}|}
		\sum_{\lambda\in \borK'}
		\sum_{\substack{t \geq 1 \\ t \mid \lambda}} 
		\sum_{h^+} \left(\frac{t}{2i}\right)^{h^+}
		\sum_{\substack{\lambdatwo \in L_0'/L \\ \prol(\lambdatwo) = \lambda/t + M}}  e\Big(t(\lambdatwo,\genUU)\Big) c(f_{\lambdatwo},q(\lambda)/t^2)
		\int_{0}^{\infty} y^{q + (p-5)/2 - h^+} 
		  \\
		&\quad \times 
		\exp\left(-\frac{2\pi y \lambda_{w^\perp}^2}{t^2} - \frac{\pi t^2}{2 y\genU_{z^\perp}^2 } \right) 
		\exp(-\Delta /8\pi y)(\polw{\aalpha,\borw}{h^+}{0})\big( \borw(\lambda/t)\big) dy
		\cdot e\big( (\lambda,\mu) \big).
		\end{align*}
		The fact that the series above is a Fourier expansion as in~\eqref{eq;genshapeFexpinvEichler} follows from the fact that the terms multiplying~$e\big((\lambda,\mu)\big)$ are invariant with respect to Eichler transformations of the form~$E(\genU,\genvec)$, with~$\genvec\in\brK\otimes\RR$.
		This follows easily from Proposition~\ref{prop;auxforFexpnohermitdom}.
	\end{proof}
		
	\section{The injectivity of the Kudla--Millson theta lift of genus~$1$}\label{sec;injKMgenus1}
	This section is devoted to the proof of the injectivity of the Kudla--Millson theta lift~$\KMliftbase$ of genus $1$ associated to indefinite lattices of general signature~$(p,q)$ that split off two (scaled) hyperbolic planes.
	This provides a generalization of~\cite{bruinier;converse} to lattices of general signature, as well as of~\cite{brfu} and~\cite{stein} to the case where the lattices are neither unimodular nor maximal.
	\begin{thm}\label{thm;injindeg1}
		Let $L$ be an even indefinite lattice of signature~$(p,q)$.
		\begin{enumerate}[label=(\roman*)]
			\item If~$L \cong K \oplus U(N) \oplus U$ for some even lattice~$K$ of signature $(p-2,q-2)$ and some positive integer~$N$, then the Kudla--Millson lift associated to $L$ is injective.
			\label{thm;injindeg1;pq>1}
			\item Let~$q=1$.
			If~$L \cong M \oplus U$ for some even lattice~$M$ of signature $(p-1,0)$ and~$M\otimes\ZZ_\primen$ splits off a hyperbolic plane over~$\ZZ_\primen$ for every prime~$\primen$, then the lift is injective.
			\label{thm;injindeg1;q=1}
			\item Let~$p=1$.
			Then the Kudla--Millson lift associated to~$L$ is identically zero.
			\label{thm;injindeg1;p=1}
		\end{enumerate}
	\end{thm}
	
	The following theorem, used by Bruinier to derive a converse theorem for Borcherds' products, is a direct consequence.
	
	\begin{cor}[{\cite[Theorem~$5.3$]{bruinier;converse}}]\label{cor;covered0}
		Assume that $L \cong K \oplus U(N) \oplus U$ for some even positive definite lattice $K$.
		Then the map $\KMliftbase\colon S^k_L \to \mathcal{Z}^2(X)$ is injective.
	\end{cor}
	
	Note that the assumption of a hyperbolic split is in general necessary to obtain an injectivity result.
	In fact, Bruinier provides examples of nontrivial cusp forms in the kernel of~$\KMliftbase$ otherwise (cf. \cite[Sec 6.1]{bruinier;converse}).

	Whenever~$L$ splits off a hyperbolic plane, i.e.~$L=M\oplus U$ for some sublattice~$M$ orthogonal to~$U$, we choose the isotropic vectors~$\genU$ and~$\genUU$ introduced in Section~\ref{sec;follBorcherds} to be the standard generators of~$U$.
	Moreover, we may assume without loss of generality that the orthogonal basis~$(\basevec_j)_j$ of~$V$ is such that
	\ba\label{eq;choicebasis1}
	\basevec_1,\dots,\basevec_{p-1},\basevec_{p+1},\dots,\basevec_{p+q-1}\quad\text{is a basis of~$M\otimes\RR$},
	\\
	\genU=\frac{\basevec_p + \basevec_{p+q}}{\sqrt{2}}\quad\text{and}\quad
	\genUU=\frac{\basevec_p - \basevec_{p+q}}{\sqrt{2}}.\qquad\quad
	\ea
	These assumptions will simplify some of the formulas provided in the previous sections.
	
	To prove Theorem~\ref{thm;injindeg1}, we need the following technical result.
	\begin{lemma}\label{lemma;technnonzeropolinprinj}
	Let $L$ be an even indefinite lattice of signature~$(p,q)$ with~$p>1$ that splits off a hyperbolic plane, i.e.~$L=\brK\oplus U$ for some sublattice~$\brK$.
	We choose~$\genU$ and~$\genUU$ as in~\eqref{eq;choicebasis1}.
	If~${\alpha_1\in\{1,\dots,p-1\}}$ and~$\aalpha=(\alpha_1,\dots,\alpha_1)$, then there exists~$g\in G$ such that
	\be\label{eq;technnonzeropolinprinj}
	\exp(-\Delta/8\pi y)\Big(
	\polw{\aalpha,\borw}{h^+}{0}
	\Big)
	\big(\borw(\lambda)\big)
	=
	\begin{cases}
	2^q, & \text{if~$h^+=q$,}\\
	0, & \text{otherwise,}
	\end{cases}
	\ee
	for every~$\lambda\in\brK\otimes\RR$.
	\end{lemma}
	\begin{proof}[Proof of Lemma~\ref{lemma;technnonzeropolinprinj}]
	We recall from Definition~\ref{def;QgenandPgen} that, under the assumptions in~\eqref{eq;choicebasis1}, for every vector $\genvec=\sum_{j=1}^{p+q}x_j\basevec_j$ in~$L\otimes\RR$ the polynomial $\pol_{\aalpha}$ is such that
	\be\label{eq;polabinspeccase}
	\pol_{\aalpha}(\genvec)=2^{q/2} x_{\alpha_1}^q.
	\ee
	We define~$g\in G$ to be the isometry interchanging~$\basevec_{\alpha_1}$ with~$\basevec_p$, mapping~$\basevec_{p+1}$ to~$-\basevec_{p+1}$, and fixing the remaining standard basis vectors.
	We remark that $g$ is an element of the stabilizer~$K$ of the base point $z_0\in \Gr(V)$.
	For this choice of~$g$, we deduce that~${\pol_{\aalpha}\big(g(\genvec)\big)=2^{q/2} x_p^q}$.

	We write~$\pol_{\aalpha}$ as in Remark~\ref{ref;howtorewrourpolgtilde}, for some homogeneous polynomials~${\polw{\aalpha,\borw}{h^+}{0}}$ of degree respectively $(q-h^+,0)$ on the vector spaces $\borw(L\otimes\RR)$.
	The choices~\eqref{eq;choicebasis1} imply that
	\bes
	(\genvec,u_{z_0^\perp})=\sum_{j=1}^{p+q}x_j(\basevec_j,\basevec_p)/\sqrt{2}=x_p/\sqrt{2},
	\ees
	hence, we deduce that
	\be\label{eq;prooftecxbeta}
	\pol_{\aalpha}\big(g(\genvec)\big)=(\genvec,u_{z_0^\perp})^q\cdot 2^{q}.
	\ee
	If we compare~\eqref{eq;prooftecxbeta} with the formula provided by Remark~\ref{ref;howtorewrourpolgtilde}, or directly using Lemma~\ref{lemma;rewritborchimpldecourpol}, we see that for this special choice of $g$ we have
	\bes
	\polw{\aalpha,\borw}{h^+}{0}\big(\borw(\genvec)\big)=\begin{cases}
	2^q, & \text{if $h^+=q$,}\\
	0, & \text{otherwise.}
	\end{cases}
	\ees
	Since~$\polw{\aalpha,\borw}{h^+}{0}$ is constant (hence harmonic) for every~$h^+$, we immediately deduce~\eqref{eq;technnonzeropolinprinj}.
	\end{proof}
	
	We are now ready to prove the main result of this section.
	
	\begin{proof}[Proof of Theorem~\ref{thm;injindeg1}]
	We begin with the cases~\ref{thm;injindeg1;pq>1} and~\ref{thm;injindeg1;q=1}, where~$L \cong M \oplus U$ splits off a hyperbolic plane~$U$.
	Recall that without loss of generality we may choose~$\genU$ and~$\genUU$ as in~\eqref{eq;choicebasis1}.
	Let~$f\in S^k_{1,L}$ be such that~$\KMliftbase(f)$ vanishes.
	We will prove that this implies $f=0$.	
	Recall that we may compute $\KMliftbase(f)$ as 
	\be\label{eq;injrec}
	\KMliftbase(f)=\sum_{\aalpha}\Big(
	\int_{\SL_2(\ZZ)\backslash\HH}y^{k} \langle f(\tau), \Theta_L(\tau,g,\pol_{\aalpha}) \rangle \frac{dx\,dy}{y^2}
	\Big) \cdot g^*(\omega_{\aalpha}),
	\ee
	for every $z\in\Gr(V)$, and every $g\in G$ such that $g$ maps $z$ to $z_0$; see~\eqref{eq;KMlambdaaO}.
	Since the elements~$\omega_{\aalpha}$ are linearly independent in~${\bigwedge}^q(\mathfrak{p})^*$, we deduce from~\eqref{eq;injrec} that~${\KMliftbase(f)=0}$ if and only if all defining integrals of the Kudla--Millson lift are zero, that is
	\be\label{eq;injreccoeffab}
	\int_{\SL_2(\ZZ)\backslash\HH}y^{k}\langle f(\tau),\Theta_L(\tau,g,\pol_{\aalpha})\rangle\frac{dx\,dy}{y^2}=0,\qquad\text{for every $\aalpha$ and $g\in G$.}
	\ee

	As a complex-valued function on~$G$, the defining integral~\eqref{eq;injreccoeffab} of the Kudla--Millson lift admits a Fourier expansion.
	By Theorem~\ref{thm;Fourexpofcoef}, this is
	\ba\label{eq;Fexpfininproofinj}
	\int_{\SL_2(\ZZ)\backslash\HH}&\frac{y^{k}}{\sqrt{2} \lvert \genU_{z^\perp} \rvert} \langle F_{\brK}(\tau; 0,0) , \Theta_{\brK}(\tau,\borw,\polw{\aalpha,\borw}{0}{0}) \rangle \frac{dx\,dy}{y^2} \\
		&+
		\sum_{\borgamma \in \brK'/\brK} \sum_{\lambda \in \brK + \borgamma} \frac{\sqrt{2}}{\lvert \genU_{z^\perp} \rvert}\sum_{h^+=0}^{q} 
		\sum_{\substack{t\ge 1 \\t\mid\lambda}} \left( \frac{t}{2i} \right)^{h^+} c(f_{\borgamma/t},q(\lambda)/t^2)
		\\
		&\times \int_0^{+\infty} \hspace{-.2em} y^{(p-5)/2+q-h^+} \exp\bigg(-\frac{2\pi y \lambda_{w^\perp}^2}{t^2} - \frac{\pi t^2}{2y\genU_{z^\perp}^2} \bigg) 
		\\ 
		&\times \exp(-\Delta/8\pi y)\big(\polw{\aalpha,\borw}{h^+}{0}\big)\big(\borw(\lambda/t)\big)dy \cdot e\big((\lambda,\mu)\big). 
	\ea
	The first summand in~\eqref{eq;Fexpfininproofinj} is the constant term of the Fourier expansion.
	We deduce from~\eqref{eq;injreccoeffab} that the Fourier coefficients of the Fourier expansion~\eqref{eq;Fexpfininproofinj} are all zero.	
	We use this to show that $c_n(f_\borgamma)=0$ for every~$\borgamma\in\brK'/\brK$ and every~$n\in\ZZ + q(\borgamma)$, that is, the cusp form $f$ is zero.
	
	We work by induction on the divisibility of all $\lambda\in\brK+\borgamma$ such that $q(\lambda)>0$, for every~$\borgamma \in \brK'/\brK$.
	Suppose that such a~$\lambda$ is \emph{primitive}, that is, the only integer $t\ge1$ dividing~$\lambda$ is~$t=1$.
	The fact that the Fourier coefficient of~\eqref{eq;Fexpfininproofinj} associated to $\lambda$ equals zero is equivalent to
	\ba\label{eq;injreccorFcoefffc}
	\frac{\sqrt{2} c(f_{\borgamma},q(\lambda))}{\lvert \genU_{z^\perp} \rvert}
	&
	\sum_{h^+=0}^q (2i)^{-h^+}\int_0^{+\infty} y^{(p-5)/2+q-h^+}\cdot
	\exp\bigg(-2\pi y \lambda_{w^\perp}^2 -
	\frac{\pi}{2yu_{z^\perp}^2}
	\bigg)\\
	&\times
	\exp(-\Delta/8\pi y)\big(\polw{\aalpha,\borw}{h^+}{0}\big)\big(\borw(\lambda)\big)
	dy=0,
	\ea
	for every~$\aalpha$ and~$g\in G$.
	Let~${\alpha_1\in\{1,\dots,p-1\}}$.
	By Lemma~\ref{lemma;technnonzeropolinprinj}, there exists $g\in G$ such that for the choice~$\aalpha=(\alpha_1,\dots,\alpha_1)$ we may simplify the sum over~$h^+$ on the left-hand side of~\eqref{eq;injreccorFcoefffc} and deduce that
	\be\label{eq;simplifwithtechlemma}
	c(f_{\borgamma},q(\lambda))\frac{i^{-q}\sqrt{2}}{\lvert \genU_{z^\perp} \rvert}\int_0^{+\infty} y^{(p-5)/2}\cdot
	\exp\Big(-2\pi y \lambda_{w^\perp}^2 -
	\frac{\pi}{2yu_{z^\perp}^2}
	\Big)
	dy=0.
	\ee
	Since the integral appearing in~\eqref{eq;simplifwithtechlemma} is \emph{strictly positive}, we deduce that~$c(f_{\borgamma},q(\lambda))=0$, which concludes the first step of the induction.
	
	Suppose next that $c(f_{\borgamma},q(\lambda'))=0$ for every~$\borgamma\in\brK'/\brK$ and every~$\lambda'\in\brK+\borgamma$ divisible by at most~$s$ positive integers.
	Let $\lambda\in\brK+\borgamma$ be such that it is divisible by~$s+1$ integers~${1<d_1<\dots <d_s}$.
	Since~${c(f_{\borgamma/d_j},q(\lambda/d_j))=0}$ for every~$j=1,\dots,s$ by induction, we may simplify the formula of the Fourier coefficient associated to $\lambda$ of the Fourier expansion~\eqref{eq;Fexpfininproofinj} again to~\eqref{eq;injreccorFcoefffc}, where this time $\lambda$ is non-primitive.
	Since the primitivity of~$\lambda$ is irrelevant in Lemma~\ref{lemma;technnonzeropolinprinj}, we may deduce~$c(f_{\borgamma},q(\lambda))=0$ with the same procedure used for the case of primitive~$\lambda$.
	Summarizing, we proved that if~$\KMliftbase(f)=0$, then
	\be\label{eq;coefvanbtna}
	c(f_{\borgamma},q(\lambda))=0\qquad\text{for every~$\borgamma\in\brK'/\brK$ and every~$\lambda\in\brK + \borgamma$.}
	\ee	
	
	Since~$\brK'/\brK\cong L'/L$, to conclude the proof it is enough to show that for every~${\xi\in L+\borgamma}$, there exists~${\lambda\in\brK+\borgamma}$ such that~$q(\xi)=q(\lambda)$.
	
	Under the assumptions of~\ref{thm;injindeg1;pq>1}, hence if we assume that~$\brK$ splits off~$U(N)$, then that can be shown using the newform theory introduced in~\cite[Section~$3.1$]{bruinier;converse} as in the last part, i.e. Part 3, of the proof of~\cite[Theorem~$5.3$]{bruinier;converse}.
	In fact, the newform theory and the last part of the cited theorem do not depend on the signature of~$L$.
	
	Under the assumptions of~\ref{thm;injindeg1;q=1} it is enough to apply~\cite[Corollary~$2.7$]{kieferzuffetti} with~$r=1$. 
	
	We conclude the proof of Theorem~\ref{thm;injindeg1} with~\ref{thm;injindeg1;p=1}.
	If~$f\in S^k_{1,L}$, then the lift~$\KMliftbase(f)$ admits a Fourier expansion as provided in Theorem~\ref{thm;Fourexpofcoef}.
	The sublattice~$\brK$ is \emph{negative definite}, in fact, its signature is~$(0,q-1)$.
	This implies that every~$\lambda\in\brK'$ has non-positive norm.
	By Lemma~\ref{lemma;rewritborchimpldecourpol} the polynomial~$\polw{\aalpha,\borw}{0}{0}$ is zero for every~$g\in G$, hence the constant coefficient of the Fourier expansion~\eqref{eq;Iaalpha} vanishes.
	Since for every~$\borgamma\in L'/L$ the Fourier coefficients of~$f_\borgamma$ associated to non-positive rational numbers are zero, we deduce that the Fourier coefficients of~\eqref{eq;Iaalpha} are all zero.
	Equivalently, the lift~$\KMliftbase(f)$ vanishes.
	\end{proof}
	
	\section{The Funke--Millson twist of the Kudla--Millson lift}\label{sec;caseofsymmpowers}
	In~\cite{fm;cycleswith} Funke and Millson constructed twists of the Kudla--Millson Schwartz function with respect to symmetric algebras.
	These are of interest for the computation of cohomology groups with local coefficients on locally symmetric spaces.
	The theta lifts induced by such twisted Schwartz functions have been proved to be injective by Bruinier and Funke~\cite{brfu} and Stein~\cite{stein} under strong hypothesis on the lattice~$L$.
	In this section we illustrate how to generalize such results following the ideas of Section~\ref{sec;injKMgenus1}.
	\\
	
	Let~$L$ be an even indefinite lattice of signature~$(p,q)$, and let~$V = L \otimes \RR$.
	We fix a non-negative integer~$\ell$.
	Funke and Millson~\cite{fm;cycleswith} introduced a twist of the Kudla--Millson Schwartz function~$\varphi_{\text{\rm KM}} $ with respect to the~$\ell$-th symmetric power of the space~$V$.
	
	Recall that the symmetric algebra~$\Sym(V)$ is the quotient of the tensor algebra~$T(V)$ by the two-sided ideal generated by the vectors of the form~$v\otimes w - w\otimes v$, where~$v,w\in V$.
	Let~$(\basevec_j)_j$ be an orthogonal basis of~$V$ chosen as in Section~\ref{sec;compDp2}.
	The tensor algebra $T(V)$ is generated by tensor products of the basis vectors~$\basevec_j$. 
	We denote by~$T^\ell(V)$ the subspace of~$T(V)$ given by tensors of degree~$\ell$, i.e.\ the subspace generated by tensor products of exactly~$\ell$ of the basis vectors~$(\basevec_j)_j$. 
	Its image in the quotient algebra~$\Sym(V)$ is the~$\ell$-th symmetric power of~$V$ and is denoted by~$\Sym^{\ell}(V)$. 
	
	\subsection{Twists of the Kudla--Millson Schwartz function}
	
	We recall here the construction of the twisted Schwartz form~$\varphi_{q,\ell}$ introduced by Funke and Millson~\cite{fm;cycleswith}.
	\\
	
	The form~$\varphi_{q,\ell}$ is an element of~$\big[\mathcal{S}(V)\otimes \mathcal{A}^q(\hermdom)\otimes\Sym^\ell(V)\big]^G$.
	Here the action of~$G$ on~$\Sym^\ell(V)$ is the one induced by the standard action of~$G$ on~$T(V)$.
	As for the case of~$\ell=0$ considered in Section~\ref{sec;compDp2}, we construct~$\varphi_{q,\ell}$ as a~$K$-invariant element of~${\mathcal{S}(V)\otimes{\bigwedge}^q(\mathfrak{p}^*)\otimes\Sym^\ell(V)}$ and then spread it to the whole space~$\hermdom$. 
	The form~$\varphi_{q,\ell}$ is defined by applying~$\ell$ times the operator
	\[
	\frac{1}{2}\sum_{\indextwo=1}^p\bigg(x_\indextwo-\frac{1}{2\pi}\frac{\partial}{\partial x_\indextwo}\bigg) \otimes 1 \otimes A_{\basevec_{\indextwo}}
	\]
	to the~$K$-invariant Kudla--Millson Schwartz function~$\varphi_{\textrm{KM}}$.
	Here~$A_{\basevec_{\indextwo}}$ denotes the multiplication by~$\basevec_{\indextwo}$ in~$\Sym(V)$.
	
	We now rewrite~$\varphi_{q,\ell}$ explicitly.
	In what follows, we denote by~$\aalpha$ and~$\bbeta$ tuples with values in~$\{1, \hdots, p\}$ and of length~$q$ and~$\ell$, respectively.
	To simplify the notation, we define
	\be\label{eq;newobjts}
	\omega_{\aalpha} \coloneqq \bigwedge_{j= 1}^{q} \omega_{\alpha_j, j}\in{\bigwedge}^q (\mathfrak{p}^*)
	\qquad
	\text{and}
	\qquad
	e_{\bbeta} \coloneqq \prod_{j = 1}^{\ell} e_{\beta_j} \in \Sym^{\ell}(V).
	\ee
	By inserting~$\varphi_{\textrm{KM}}$ as computed in~\eqref{eq;KMschwrew}, we may deduce that
	\begin{align*}
	\varphi_{q,\ell}
	&=
	\frac{1}{2^\ell}
	\sum_{\bbeta}
	\bigg(
	\prod_{j=1}^\ell\bigg(
	x_{\indextwo_j} - \frac{1}{2\pi}\frac{\partial}{\partial x_{\indextwo_j}}
	\bigg)
	\otimes 1 \otimes \prod_{j=1}^\ell A{\basevec_{\indextwo_j}}
	\bigg)
	\varphi_{\textrm{KM}}
	\\
	&=\frac{1}{2^{\ell+q/2}}
	\sum_{\aalpha,\bbeta}
	\prod_{j=1}^\ell
	\prod_{\mu=1}^q
	\bigg(
	x_{\indextwo_j} - \frac{1}{2\pi} \frac{\partial}{\partial x_{\indextwo_j}}
	\bigg)
	\bigg(
	x_{\alpha_{\mu}} - \frac{1}{2\pi}\frac{\partial}{\partial x_{\alpha_{\mu}}}
	\bigg)\varphi_0
	\otimes\omega_{\aalpha}
	\otimes\basevec_{\bbeta}.
	\end{align*}

	Next, we rearrange the sums on the right-hand side of the formula above in order to obtain a sum of linearly independent vectors. 
	The differential operator therein depends solely on the count of values attained by the entries of~$\aalpha$ and $\bbeta$; see Remark~\ref{rem;KMsFunctionpart}.
	Analogously to what we did in Section~$2$ for~$\aalpha$, we associate a counting tuple~$\psi(\bbeta) \coloneqq \overline{\beta} = (\ell_1, \hdots, \ell_p)$
	also to~$\bbeta$, where~$\ell_j$ is the number of times~$j$ appears as an entry of the tuple~$\bbeta$.
	
	We remark that for~$\bbetau$ given, every~$\bbeta \in \psi^{-1}(\bbetau)$ gives rise to the same~$\basevec_{\bbeta} \in \Sym^{\ell}(V)$. 
	This prompts us to modify the vectors defined in~\eqref{eq;newobjts} as
	\[
	\omega^{\overline{\alpha}} 
	\coloneqq \sum_{ \aalpha \in \psi^{-1}(\aalphau) } \omega_{\aalpha} 
	\in{\bigwedge}^{q}(\mathfrak{p}^*)
	\qquad
	\text{and}
	\qquad
	e^{\overline{\beta}} 
	\coloneqq \sum_{\bbeta \in \psi^{-1}(\bbetau)} e_{\bbeta} 
	= \frac{\ell!}{\prod_{j} \ell_j!} \cdot \prod_{j = 1}^{\ell} e_j^{\ell_j} \in \Sym^{\ell}(V).
	\] 
	
	Recall from Definition \ref{def;QgenandPgen} that 
	\be\label{eq;Qbbetau}
	\Gpol^{\bbetau}(\underline{x}) 
	= \frac{1}{(4\pi)^{q/2}}\prod_{j=1}^{p} H_{\ell_j}\big(\sqrt{2\pi}x_{j}\big) 
	\qquad\text{and}\qquad
	\pol^{\bbetau}(\underline{x})
	= 
	2^{q/2}\prod_{j=1}^{p} x_{j}^{\ell_j}.
	\ee
	We interpret~$\aalphau$ and~$\bbetau$ as elements of~$l^1(\{1, \hdots, p\})$ 
	so that we may add them and rewrite~$\varphi_{q,\ell}$ by~\eqref{eq;melpanh} as
\begin{align*}
	\varphi_{q,\ell}
	&=\frac{1}{2^{\ell/2}}
	\sum_{\substack{\overline{\alpha} \\ \|\aalphau\|_1 = q}} 
	\sum_{\substack{\overline{\beta} \\ \|\bbetau\|_1 = \ell}}
	\Gpol^{\overline{\alpha}+\overline{\beta}}\varphi_0
	\otimes
	\omega^{\overline{\alpha}}
	\otimes
	\basevec^{\overline{\beta}}.
\end{align*}
	The vectors~$\omega^{\aalphau} \otimes e^{\bbetau}$ are linearly independent.
	We group these such that the parameters $\aalphau$ and~$\bbetau$ add up to the same $p$-tuple $\ggammau$ of non negative integers with $\|\ggammau\|_1 = q + \ell$.
	We also include the factor $2^{-\ell/2}$ into these vectors, so that we obtain new linearly independent vectors~${\mathfrak{b}^{\overline{\gamma}} \in {\bigwedge}^{q}(\mathfrak{p}^*) \otimes \Sym^{\ell}(V)}$.
	With this notation, we have 
	\begin{align}\label{eq;varphiqlfinal}
	\varphi_{q,\ell}
	&=
	\sum_{\substack{\overline{\gamma} \\  \|\overline{\gamma}\|_1 = q + \ell}} 
	\Gpol^{\overline{\gamma}}\varphi_0 
	\otimes
	\mathfrak{b}^{\overline{\gamma}}.
	\end{align}

	\subsection{The injectivity of the lift}
	
	We replace the weight~$k=(p+q)/2$ considered in Section~\ref{sec;injKMgenus1} by the shifted weight~$\kappa \coloneqq k + \ell$.
	We may then realize the theta function associated to~$\varphi_{q,\ell}$ as	
	\ba\label{eq;ThetaKMql}
	\Theta(\tau,z,\varphi_{q,\ell}) 
	&=
	y^{-\kappa/2} \sum_{\borgamma \in L'/L} \sum_{\lambda\in L + \borgamma}\Big(\wrep(g_\tau)\varphi_{q,\ell}\Big)(\lambda,z)
	\mathfrak{e}_\borgamma
	\\
	&= 
	\sum_{\substack{\overline{\gamma} \\ \|\ggammau\|_1 = q + \ell}}  \underbrace{y^{-\kappa/2} \sum_{\borgamma \in L'/L} \sum_{\lambda\in L + \borgamma} \Big(\wrep(g_\tau)(\Gpol^{\overline{\gamma}}\varphi_0)\Big)\big( g(\lambda)\big) \basee_\borgamma}_{\eqqcolon \FFg(\tau,g)}
	\otimes 
	\,
	g^*\big( \mathfrak{b}^{\overline{\gamma}} \big).
	\ea
	If~$\ell=0$, then the previous expression is a rewriting of~\eqref{eq;KMdeg1recallexpl},  made by grouping the summands in a different way.
	The theta function~$\Theta(\tau,z,\varphi_{q,\ell})$ behaves as a non-holomorphic modular form of weight~$\kappa$ with respect to~$\tau\in\HH$, and is a closed differential form on~$\hermdom$ with respect to the variable~$z$.
	We remark that the auxiliary function~$\FFg(\tau,g)$ appearing on the right-hand side of~\eqref{eq;ThetaKMql} equals~$\Theta_L(\tau, g, \pol^{\overline{\gamma}})$. This may be verified as in the proof of  Proposition~\ref{lemma;FabwithBorch}. 
	We deduce that
	\be\label{eq;feqtwisttheta}
	\Theta(\tau,z,\varphi_{q,\ell})
	=
	\sum_{\substack{\overline{\gamma} \\ \|\ggammau\|_1 = q + \ell}}
	\Theta_L(\tau, g, \pol^{\overline{\gamma}})
	\otimes
	g^*\big( \mathfrak{b}^{\overline{\gamma}} \big).
	\ee
	
	Let~$\locsys{\ell}{V}$ be the local system on~$\hermdom$ associated to~$\Sym^\ell(V)$.
	We denote with the same symbol also its push-forward sheaf on~$X$ under the quotient map.
	We denote by~$\mathcal{Z}^q(X, \locsys{\ell}{V})$ the space of~$\locsys{\ell}{V}$-valued closed~$q$-forms on~$\hermdom$.

	\begin{defi}
	The Kudla--Millson--Funke lift~$\KMliftbasel$ is the linear map
	\[
	\KMliftbasel \colon S_L^{\kappa} \longrightarrow \mathcal{Z}^q(X, \locsys{\ell}{V}),
	\qquad
	f\longmapsto \int_{\SL_2(\ZZ)\backslash\HH} y^\kappa \langle f(\tau) , \Theta(\tau,z,\varphi_{q,\ell}) \rangle \frac{dx\,dy}{y^2}.
	\]
	\end{defi}
	We remark that by~\eqref{eq;feqtwisttheta} it is possible to rewrite the lift~$\KMliftbasel(f)$ in terms of Borcherds' theta functions as
	\be\label{eq;tliftwithBtf}
	\KMliftbasel(f)
	=
	\sum_{\substack{\overline{\gamma} \\ \|\ggammau\|_1 = q + \ell}}
	\underbrace{\int_{\SL_2(\ZZ)\backslash\HH} y^\kappa \langle f(\tau) , \Theta_L(\tau, g, \pol^{\overline{\gamma}}) \rangle \frac{dx\,dy}{y^2}}_{\eqqcolon\intfunctg(g)}
	\cdot 
	\,
	g^*\big( \mathfrak{b}^{\overline{\gamma}} \big).
	\ee
	In analogy with Section~\ref{sec;theunfofKMgen1}, we call the auxiliary functions~$\intfunctg\colon G\to\CC$ appearing in~\eqref{eq;tliftwithBtf} the \emph{defining integrals} of the lift~$\KMliftbasel(f)$.
	By Theorem~\ref{thm;Fourexpofcoef}, the defining integrals admit a Fourier expansion of the form
	\ba\label{eq;Iggammau}
	\intfunctg(g) 
	=& 
	\int_{\SL_2(\ZZ)\backslash\HH}\frac{y^{\kappa}}{\sqrt{2} \lvert\genU_{z^\perp}\rvert} \langle F_{\brK}(\tau; 0,0) , \Theta_{\brK}(\tau,\borw,\polw{\borw}{0}{0}^{\ggammau}) \rangle \frac{dx\,dy}{y^2} \\
	&+\sum_{\lambda \in \borK'} \sum_{\substack{t \geq 1 \\ t \mid \lambda}} \frac{\sqrt{2}}{|\genU_{z^\perp}|} 
	\sum_{h^+} \left(\frac{t}{2i}\right)^{h^+} \sum_{\substack{\lambdatwo \in L_0'/L \\ \pi(\lambdatwo) = \lambda/t + M}}  e\Big(t(\lambdatwo,\genUU)\Big) c_{q(\lambda)/t^2}(f_{\lambdatwo})  \\
	& \times \int_{0}^{\infty} y^{q + \ell + (p - 5)/2 - h^+} 
	\exp\left(-\frac{2\pi y \lambda_{w^\perp}^2}{t^2} - \frac{\pi t^2}{2 y\genU_{z^\perp}^2 } \right) 
	\\
	& \times  \exp(-\Delta /8\pi y)(\polw{\borw}{h^+}{0}^{\ggammau})\big( \borw(\lambda/t)\big) dy \cdot e\Big( (\lambda,\mu) \Big). \\ 
	\ea
	It is easy to see that the proof of the injectivity of~$\KMliftbasel$, i.e.\ Theorem~\ref{thm;injindeg1}, works also for~$\KMliftbasel$. 
	\begin{thm}\label{thm;injindeg1_l}
		Let $L$ be an even indefinite lattice of signature~$(p,q)$. 
		\begin{enumerate}[label=(\roman*)]
			\item If~$L \cong K \oplus U(N) \oplus U$ for some even lattice~$K$ and some positive integer~$N$, then the lift~$\KMliftbasel$ is injective.
			\label{thm;injindeg1;pq>1_l}
			\item Let~$q=1$.
			If~$L \cong M \oplus U$ for some positive definite even lattice $M$ and~$M\otimes\ZZ_\primen$ splits off a hyperbolic plane for every prime~$\primen$, then the lift $\KMliftbasel$ is injective.
			\label{thm;injindeg1;q=1_l}
			\item Let~$p=1$.
			Then the lift~$\KMliftbasel$ associated to~$L$ is identically zero.
			\label{thm;injindeg1;p=1_l}
		\end{enumerate}
	\end{thm}
	\begin{proof}[Sketch of the proof]
		Let~$f \in S_L^{\kappa}$ be a cusp form such that~$\KMliftbasel(f)=0$. 
		Since the vectors~$g^*\big( \mathfrak{b}^{\overline{\gamma}} \big)$ appearing on the right-hand side of~\eqref{eq;tliftwithBtf} are linearly independent for every~$g\in G$, we deduce that the defining integrals~$\intfunctg(g)$ vanish for every~$\overline{\gamma}$ and~$g$. 
		The proof now proceeds in the same way as in Theorem~\ref{thm;injindeg1}. 
	\end{proof}
	We conclude this section by showing that the injectivity results of Bruinier--Funke~\cite{brfu} and Stein~\cite{stein} are covered by our main result.
	Let~$r_0$ denote the Witt rank of~$L$.
	\begin{cor}[{\cite[Corollary~$4.12$]{brfu}}]\label{cor;covered1}
		Assume that $p + q > \max(4,3+r_0)$, $p>1$ and that~$q + \ell$ is even.
		If~$L$ is unimodular, then the lift $\KMliftbasel$ is injective.
	\end{cor}
	\begin{cor}[{\cite[Theorem~$1.1$]{stein}}]\label{cor;covered2}
		Let $p + q > \max(6,2\ell-2,3 + r_0)$. 
		Furthermore, assume that $q + \ell$ and $k + \ell$ are even.
		If~$L'/L$ is anistropic, then the lift~$\KMliftbasel$ is injective. 
	\end{cor}
		
	\begin{proof}[Proof of Corollaries~\ref{cor;covered1} and~\ref{cor;covered2}]
		Recall that the lattice~$L$ is maximal if and only if its discriminant~$L'/L$ is anisotropic. 
		Clearly, unimodular lattices are maximal so that we may reduce to this case and assume further that $\rank L \geq 7$. 
		By Meyer's Theorem, every indefinite quadratic space~$V$ over~$\QQ$ of dimension at least~$5$ splits a hyperbolic plane. In particular, $L \otimes \QQ$ splits a hyperbolic plane and by~\cite[Theorem~$14.12$]{Kneser2002} this transfers to~$L$, so that~$L = \borK \oplus U$ for some maximal lattice~$M$. 
		\\
		Consequently, if~$\min\{p,q\} \geq 2$ then~$L$ splits off two hyperbolic planes and Theorem~\ref{thm;injindeg1_l}~(i) applies. 
		In case~$q = 1$, we have $\rank \borK \geq 5$ so that $\borK \otimes \QQ_\primen$ splits a hyperbolic plane
		for every finite place~$\primen$ (cf.~\cite[Theorem~$16.3$]{Kneser2002}). 
		Again, by~\cite[Theorem~$14.12$]{Kneser2002} the same is true for $\borK \otimes \ZZ_{\primen}$, 
		so that Theorem~\ref{thm;injindeg1_l}~(ii) applies.
	\end{proof}
	
	\printbibliography
\end{document}